\documentclass[times,sort&compress,3p]{elsarticle}
\usepackage[labelfont=bf]{caption}

\makeatletter
\def\ps@pprintTitle{%
  \let\@oddhead\@empty
  \let\@evenhead\@empty
  \let\@oddfoot\@empty
  \let\@evenfoot\@oddfoot
}
\makeatother

\usepackage{amsmath,amsfonts,amssymb,amsthm,booktabs,color,epsfig,graphicx,hyperref,url}
\usepackage{mathtools,bm,bbm,comment}

\theoremstyle{plain}

\newtheorem{proposition}{Proposition}

\newtheorem{corollary}{Corollary}

\theoremstyle{definition}
\newtheorem{definition}{Definition}
\newtheorem{remark}{Remark}
\newtheorem{example}{Example}

\newcommand*{\Bern}{\operatorname{Bern}}

\newcommand*{\N}{\operatorname{N}}

\newcommand*{\Unif}{\operatorname{Unif}}

\newcommand*{\IN}{\mathbb{N}}

\newcommand*{\IR}{\mathbb{R}}

\newcommand{\argmax}{\mathop{\rm argmax}\limits}
\newcommand{\argmin}{\mathop{\rm argmin}\limits}

\newcommand*{\Cov}{\operatorname{Cov}}

\newcommand*{\E}{\mathbb{E}}
\newcommand*{\esssup}{\operatorname*{esssup}}
\newcommand*{\essinf}{\operatorname*{essinf}}

\newcommand*{\Prob}{\Pr}

\newcommand*{\rd}{\mathrm{d}}

\newcommand*{\Var}{\operatorname{Var}}

\newcommand{\bp}{\bm{p}}

\newcommand{\bX}{\bm{X}}

\newcommand{\bz}{\bm{z}}

\newcommand{\bmu}{\bm{\mu}}

\newcommand{\bzero}{\bm{0}}

\newcommand{\bone}{\bm{1}}

\newcommand*{\aseq}{\ \smash{\omu{\text{\tiny{a.s.}}}{=}{}}\ }
\newcommand*{\deq}{\ \smash{\omu{\text{\tiny{d}}}{=}{}}\ }
\newcommand{\darrow}{\stackrel{d}{\longrightarrow}}

\renewcommand*{\i}{{-1}}

\newcommand*{\iidsim}{\ \smash{\omu{\text{\tiny{iid}}}{\sim}{}}\ }

\newcommand{\ou}[3]{%
  \mathrel{%
    \vcenter{\offinterlineskip
      \ialign{##\cr$#1$\cr\noalign{\kern-#3}$#2$\cr}%
    }%
  }%
}
\newcommand*{\omu}[3]{\underset{#3}{\overset{#1}{#2}}}

\begin{document}

\begin{frontmatter}

\title{Comparison of correlation-based measures of concordance in terms of asymptotic variance}

\author[1]{Takaaki Koike\corref{mycorrespondingauthor}}
\author[2]{Marius Hofert}

\address[1]{Graduate School of Economics, Hitotsubashi University,
Naka, Kunitachi, Tokyo, Japan.}
\address[2]{Department of Statistics and Actuarial Science,
Faculty of Science,
The University of Hong Kong,
Pokfulam, Hong Kong.}

\cortext[mycorrespondingauthor]{Corresponding author. Email address: \url{takaaki.koike@r.hit-u.ac.jp}}

\begin{abstract}
  We compare measures of concordance that arise as Pearson's linear correlation
  coefficient between two random variables transformed so that they follow the
  so-called concordance-inducing distributions.  The class of such transformed
  rank correlations includes Spearman's rho, Blom\-qvist's beta and van der
  Waerden's coefficient.  When only the standard axioms of
  measures of concordance are required, it is not always clear which transformed
  rank correlation is most suitable to use.  To address this question, we
  compare measures of concordance in terms of their best and worst asymptotic
  variances of some canonical estimators over a certain set of dependence
  structures.  A simple criterion derived from this approach is that
  concordance-inducing distributions with smaller fourth moment are more
  preferable.  In particular, we show that Blomqvist's beta is the optimal
  transformed rank correlation in this sense, and
  Spearman's rho outperforms van der Waerden's coefficient.  Moreover, we find
  that Kendall's tau, although it is not a transformed rank correlation of that nature, shares
  a certain optimal structure with Blomqvist's beta.
  \end{abstract}
  

\end{frontmatter}

\section{Introduction}\label{sec:introduction}

A pair of random variables is said to be more concordant (discordant) if large values of one random variable are more likely to correspond to large (small) values of the other random variable.
When the random variables are continuous, concordance and discordance are the properties of their (bivariate) $copula$ $C$, which is a bivariate distribution function with standard uniform univariate marginal distributions.
A measure of concordance quantifies concordance or discordance of $C$ by a single number $\kappa$ in $[-1,1]$; see Definition~\ref{def:axioms:MOC} below.

In this paper we consider $G$-transformed rank correlations, a subclass of measures of concordance which can be represented as Pearson's linear correlation coefficient $\rho$ between two random variables transformed so that they follow the so-called concordance-inducing distribution $G$.
For a distribution function $G$ on $\IR$, the $G$-transformed rank correlation of a copula $C$ is given by
\begin{align*}
\kappa_G(U,V)=\rho(G^\i(U),G^\i(V)), \quad (U,V)\sim C,
\end{align*}
where $G^\i$ is the quantile function of $G$.
This measure $\kappa_G$ satisfies the axioms of a measure of concordance~\citep{scarsini1984} by taking $G$ to be concordance-inducing; see~Definition~\ref{def:G:transformed:corr}.
This class contains popular measures of concordance, such as Spearman's rho $\rho_{\text{S}}$, Blomqvist's beta $\beta$ and van der Waerden's coefficient $\zeta$; see Example~\ref{ex:transformed:rank:correlations} below.
Such a correlation representation of a measure of concordance $\kappa$ is of great benefit to intuitively understand and explain the construction and ideas behind $\kappa$ (see~Fig.~\ref{fig:illustration}), to construct its estimators and investigate their asymptotic properties, and to analyze robustness \citep{raymaekers2021fast} and matrix compatibility \citep{hofert2019compatibility} of $\kappa$.
Moreover, \cite{koike2022matrix} showed that the class of transformed rank correlations exhausts all measures of concordance of the form $\kappa(U,V)=\rho(g_1(U),g_2(V))$ for two possibly discontinuous functions $g_1,\ g_2: (0,1)\rightarrow \IR$.

\begin{figure}[t]
  \centering
  \vspace{0mm}
  \includegraphics[width=15 cm]{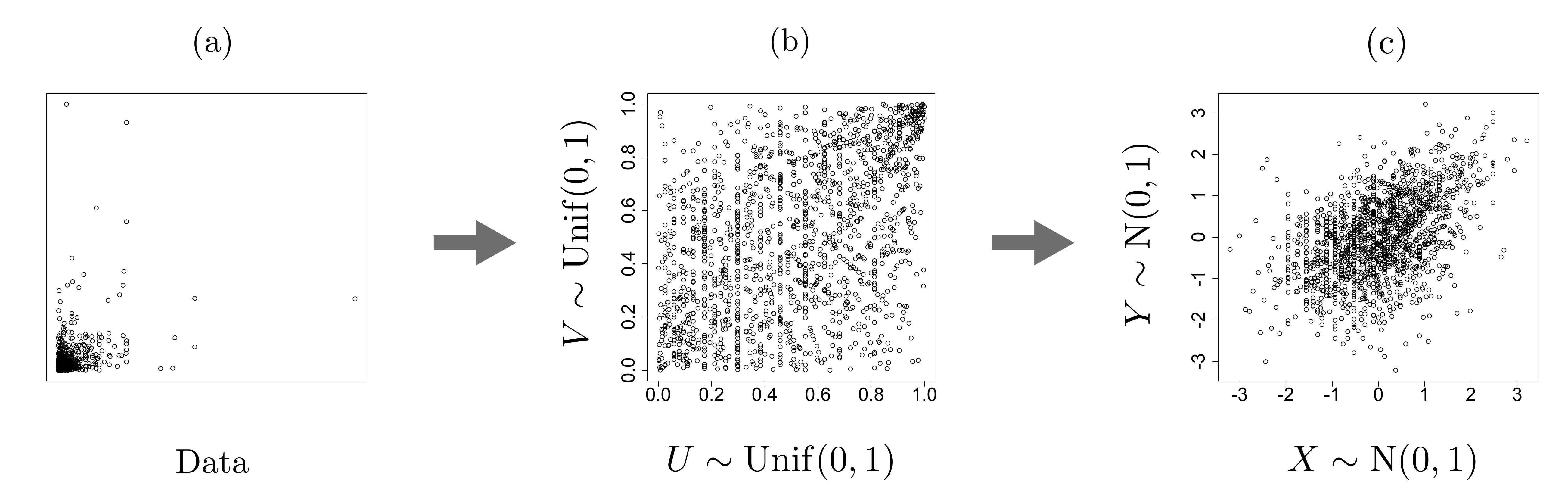}
  \vspace{0mm}
  \caption{Illustration of the process of calculating the $G$-transformed rank correlation $\kappa_G$ where $G$ is the standard normal distribution.
  The original data in (a) is first transformed so that their margins follow the standard uniform distribution.
  The transformed data in (b) is then converted so that their margins have the distribution $G$.
  Finally, the $G$-transformed rank correlation $\kappa_G$ is calculated as a Pearson's linear correlation coefficient of the transformed data in  (c).}
  \label{fig:illustration}
\end{figure}

When only the axioms of a measure of concordance are required, it is not always clear which transformed rank correlation is most suitable to use.
For the related literature, \cite{de2016comparing} compares Pearson's linear correlation coefficient and Spearman's rho by numerical experiments in terms of bias, variance and robustness to outliers.
Rank correlations are also extensively compared in \cite{tarsitano2009comparing}.
Various measures of concordance are compared in terms of their power in tests of independence; see, for example,~\cite{genest2005locally}, \cite{luigi1999asymptotic} and \cite{rodel2004linear}.

To address this natural question of how to choose the concordance-inducing distribution, the aim of the present paper is to compare $G$-transformed rank correlations in terms of their stability concerning statistical estimation.
Estimation of $\kappa_G$ is often inevitable since an explicit form of $\kappa_G(C)$ is not always available.
For the purpose of comparison, we consider a simplified setting where independent and identically distributed (i.i.d.) samples from the underlying copula $C$ are available, and the $G$-transformed rank correlation $\kappa_G$ is estimated by the so-called canonical estimator $\hat \kappa_G$ (Definition~\ref{def:canonical:estimator:kappa:G}).
Although this simplified setting may be rarely the case, the asymptotic variance $\sigma_G^2(C)$ of the canonical estimator $\hat \kappa_G$ is obtained in a tractable form.
For an underlying copula $C$, a concordance-inducing distribution $G$ can be more preferable to another one $G'$ in terms of the stability of statistical estimation if $\sigma_{G}^2(C)\leq \sigma_{G'}^2(C)$.
Since this comparison is valid only for a specific copula $C$, we consider a set of copulas $\mathcal D$ and compare concordance-inducing distributions by the largest and smallest values of  $\sigma_G^2(C)$ over $\mathcal D$.
Namely, a concordance-inducing distribution $G$ is more preferable to another one $G'$ in $\mathcal D$ if the worst and best asymptotic variances of a canonical estimator $\hat \kappa_G$ of $\kappa_G$ on $\mathcal D$ are smaller than those of $G'$.
For such $G$ and $G'$, if $\mathcal D$ represents possible dependence structures which the analyst is interested in quantifying and comparing, she may be more willing to use $\kappa_G$ instead if $\kappa_G'$ since the former
is expected to be estimated more accurately than the latter.

\begin{figure}[t]
  \centering
  \vspace{0mm}
  \includegraphics[width=12 cm]{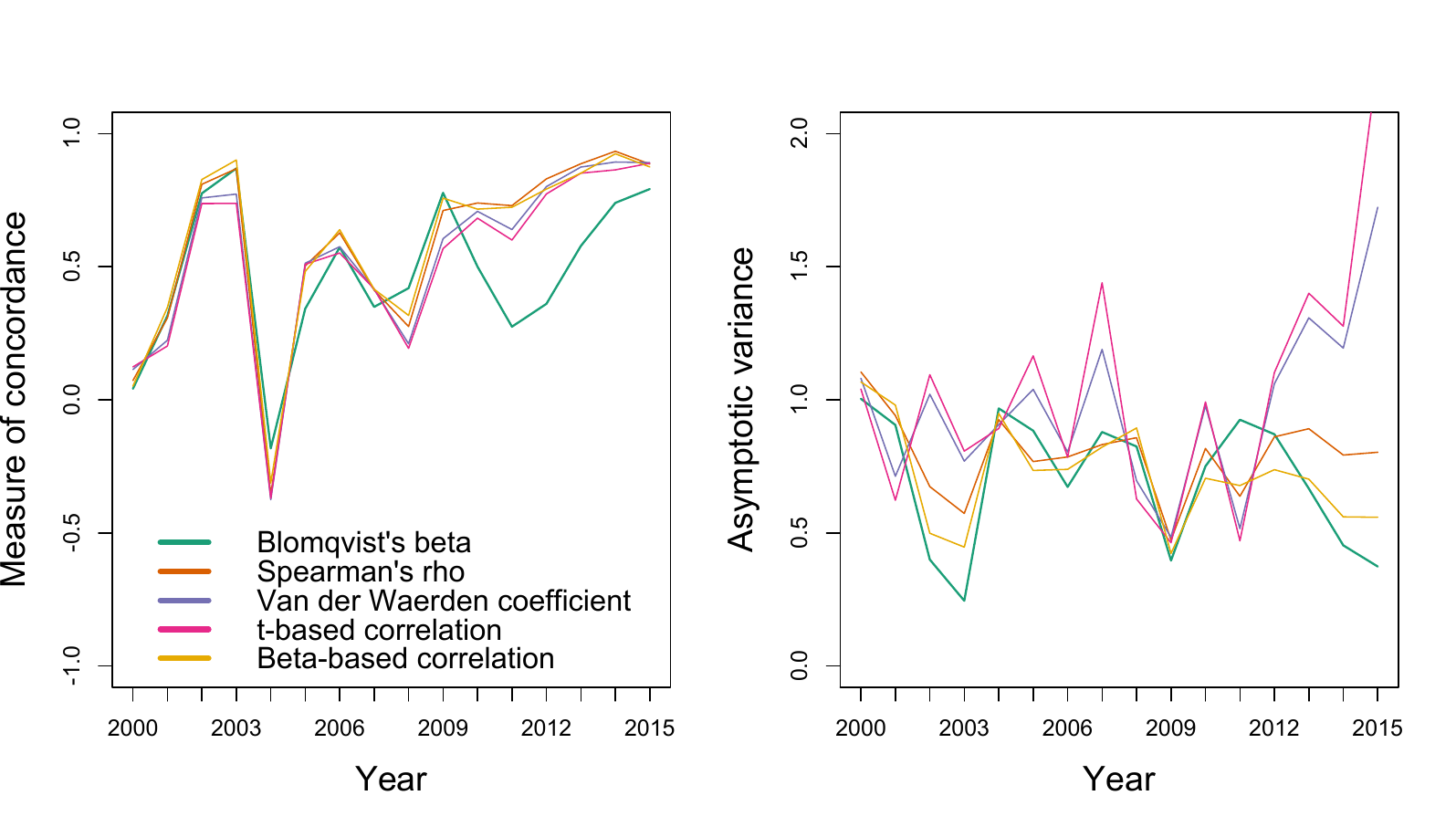}
  \vspace{0mm}
  \caption{Estimates of $G$-transformed rank correlations (left) and their asymptotic variances (right) of dependence among the daily log increments of exchange rates JPY$/$USD and CAD$/$USD filtered by skew $t$ GARCH(1,1) model each year in the period of 2000--2015.}
  \label{fig:empirical}
\end{figure}

To illustrate
the motivation of this work, let us consider the problem of quantifying
the dependence among two exchange rates JPY$/$USD and CAD$/$USD from 2000 to
2015. Following~\cite{abdullah2017modeling} and \cite{epaphra2016modeling}, we model
the daily log increments of each exchange rate in each year by a GARCH(1,1) model with skew-$t$
innovation distribution. 
We then capture the dependence among the two time
series by the copula of the bivariate standardized
residuals. Based on the filtered i.i.d.\ samples, a transformed rank correlation $\kappa_G$ is
estimated by the canonical estimator $\hat \kappa_G$ for different choices of
$G$. Other than $\beta$, $\rho_\text{S}$ and $\zeta$, a standardized Student $t$
distribution with $\nu=10$ degrees of freedom and a beta distribution with
parameter vector $(0.5,0.5)$ are also considered as concordance-inducing distribution $G$. The corresponding asymptotic variances $\sigma_G^2$ are estimated from
samples as in Section~\ref{sec:simulation:study}. The results are plotted in
Fig.~\ref{fig:empirical}, which also shows that the yearly trend of dependence is captured similarly by all the chosen measures of concordance, although some deviations are observed for $\beta$. Therefore, one
may be willing to choose measures that are more stable to estimate.
From the right plot in Fig.~\ref{fig:empirical}, $\zeta$ and $t$-based rank
correlations tend to have larger asymptotic variances than others, and $\beta$
often attains relatively small $\sigma_G^2$.  From these observations, $\beta$,
$\rho_\text{S}$ and the beta distribution-based rank correlations may be preferable over
$\zeta$ and the $t$ distribution-based rank correlation, at least in terms of the stability of
their statistical estimation.

From this approach for comparing transformed rank correlations, we derive a
criterion under some reasonable assumptions on $\mathcal D$ that a
concordance-inducing distribution $G$ with smaller variance $\Var_G(X^2)$, where
$X\sim G$, is more preferable.  Therefore, heavy-tailed concordance-inducing
distributions, such as the Student $t$ distribution, are not preferable (at least in terms of asymptotic variance) as a
choice of $G$ in comparison to a normal distribution, which leads to van der Waerden's coefficient $\zeta$.  It is also deduced
that $\rho_{\text{S}}$ outperforms $\zeta$, and the beta distribution-based
transformed rank correlations are more preferable to $\rho_{\text{S}}$ in this sense. In
particular, we prove that Blomqvist's beta is the unique optimal transformed
rank correlation attaining the optimal best and worst asymptotic variances under
certain conditions on $\mathcal D$. As stated by
\cite{schmid2007nonparametric}, one of the advantages of $\beta$ over other
measures of concordance is that it admits an explicit form if the copula can be
written explicitly.  The optimality of $\beta$ provides its additional advantage
in terms of accuracy of its estimation.

Despite the advantages of $\beta$, its major drawback is that it depends only on
the local value $C(1/2,1/2)$ of the underlying copula $C$.  Considering this
drawback, we also compare $G$-transformed rank correlations with Kendall's tau
$\tau$.  Based on the representation of $\tau$ in terms of Pearson's linear
correlation coefficient, we find that $\tau$ also attains the optimal best and
worst asymptotic variances that $\beta$ does.  Therefore, $\tau$ can be a good
alternative to $\beta$ in terms of the accuracy of its estimation even though the
optimality of $\tau$ is partly violated if the sample size required to construct
its estimator is taken into account (see
Remark~\ref{remark:comparability:kappa:g:tau} below).  Finally, in a simulation
study, we find that the choice of concordance-inducing distribution $G$ and the
strength of dependence of the underlying copula $C$ affect the asymptotic
variance of $\hat \kappa_G$ more than the model of dependence.

This paper is organized as follows.  In
Section~\ref{sec:correlation:based:mocs}, we review measures of concordance, the
class of $G$-transformed rank correlations $\kappa_G$ and their basic
properties.  In Section~\ref{sec:estimation:kappa:g:comparison} we introduce the
approach for comparing $G$-transformed rank correlations in terms of their
asymptotic variances.  A canonical estimator of $\kappa_G$ is presented in
Section~\ref{subsec:canonical:estimator:kappa:G}, and the properties of its
asymptotic variance $\sigma_G^2(C)$ are studied in
Section~\ref{subsec:properties:asymptotic:variance}.
Section~\ref{subsec:optimal:location:shift} addresses effects of location-scale
transforms of $G$ on $\sigma_G^2(C)$.  In
Section~\ref{sec:optimal:concordance:inducing:distributions}, we study optimal
best and worst asymptotic variances and their attaining concordance-inducing
distributions on $\mathcal D$.  The case when $\mathcal D$ is a set of
fundamental or Fr\'echet copulas is analyzed in
Section~\ref{subsec:asymptotic:variance:fundamental:copulas}.  The optimality of
Blomqvist's beta is proved in Section~\ref{subsec:optimality:Blomqvist}, and its
uniqueness is discussed in Section~\ref{subsec:uniqueness:optimality:beta}.
Kendall's tau and $G$-transformed rank correlations are compared in
Section~\ref{sec:comaprison:kendall:tau}.  In
Section~\ref{sec:simulation:study}, a simulation study is conducted to compare
the asymptotic variances for various parametric copulas and concordance-inducing
distributions.  Section~\ref{sec:concluding:remark} concludes this work with
discussions about directions for future research.
Proofs of the statements are given in Section~\ref{sec:proofs}.

\section{Correlation-based measures of concordance}\label{sec:correlation:based:mocs}

Let $\mathcal C_2$ be the set of all bivariate copulas, that is, all bivariate distribution functions with standard uniform marginal distributions.
We call $C' \in \mathcal C_2$ more concordant than $C \in \mathcal C_2$, denoted by $C\preceq C'$, if $C(u,v)\leq C'(u,v)$ for all $(u,v) \in [0,1]^2$.
The survival function of $C$ is denoted by $\bar C(u,v)=\Prob(U>u,V>v)$, $(u,v)\in[0,1]^{2}$, where $(U,V)\sim C$.
The comonotonicity, counter-monotonicity and independence copulas are denoted by $M(u,v)=\min(u,v)$, $W(u,v)=\max(u+v-1,0)$ and $\Pi(u,v)=uv$, $(u,v)\in [0,1]^2$, respectively.
By the Fr\'echet--Hoeffding inequalities, it holds that $W\preceq C \preceq M$ for all $C \in \mathcal C_2$.

Consider a group of transforms on $\mathcal C_{2}$:
$$\mathcal T = \{\iota,\, \nu_{1}, \,\nu_{2},\, \nu_1\circ \nu_2, \,\pi, \,\pi\circ \nu_{1},\,\pi\circ \nu_{2},\,\pi\circ \nu_{1}\circ \nu_{2}\},$$
where $\iota:\mathcal C_2 \rightarrow \mathcal C_2$ is the identity $\iota(C)=C$;
$\nu_1:\mathcal C_2 \rightarrow \mathcal C_2$ and $\nu_2:\mathcal C_2 \rightarrow \mathcal C_2$ are the partial reflections defined, respectively, by
\begin{align*}
\nu_1(C)(u,v)  =v-C(1-u,v),\quad
\nu_2(C)(u,v)  =u-C(u,1-v),\quad C \in \mathcal C_2;
\end{align*}
their composition is given by $\nu_1\circ \nu_2(C)(u,v)=u+v-1+C(1-u,1-v)$; and $\pi:\mathcal C_2 \rightarrow \mathcal C_2$ is the permutation $\pi(C)(u,v)=C(v,u)$, for $(u,v) \in [0.1]^{2}$.
Let $C_{\varphi} = \varphi(C)$ for $\varphi  \in \mathcal T$, and
denote by
\begin{alignat*}{3}
(U_{\nu_1},V_{\nu_1})&\aseq (1-U,V)\sim C_{\nu_1},&\qquad
(U_{\nu_2},V_{\nu_2})&\aseq (U,1-V)\sim C_{\nu_2},\\
(U_{\pi},V_{\pi}) &\aseq (V,U)\sim C_{\pi},&\qquad
(U_{\nu_1\circ \nu_2},V_{\nu_1\circ \nu_2})&\aseq (1-U,1-V)\sim C_{\nu_1\circ \nu_2},
\end{alignat*}
for $(U,V)\sim C$, where ``a.s.'' stands for almost surely.

For any map $\kappa:\mathcal C_2\rightarrow \IR$, we identify $\kappa(C)$ with $\kappa(U,V)$ for a random vector $(U,V)\sim C$ defined on a fixed atomless probability space $(\Omega,\mathcal A,\Prob)$.
A map $\kappa$ on $\mathcal C_2$ is called a measure of concordance if it satisfies the following seven axioms~\citep{scarsini1984}.

\begin{definition}[Axioms for measures of concordance]\label{def:axioms:MOC}
A map $\kappa:\mathcal C_2 \rightarrow \IR$ is called a measure of concordance if it satisfies the following seven axioms:
  \begin{enumerate}
  \item \emph{Domain}: $\kappa(C)$ is defined for any $C \in \mathcal C_2$;
  \item \emph{Symmetry}: $\kappa(C_{\pi})=\kappa(C)$ for any $C \in \mathcal C_2$;
  \item \emph{Monotonicity}: If $C\preceq C'$ for $C,C' \in \mathcal C_2$, then $\kappa(C)\le\kappa(C')$;
  \item \emph{Range}: $-1 \leq \kappa(C)\leq 1$ for any $C \in \mathcal C_2$, $\kappa(M)=1$ and $\kappa(W)=-1$;
  \item \emph{Independence}: $\kappa(\Pi)=0$;
  \item \emph{Change of sign}: $\kappa(C_{\nu_{1}}) =\kappa(C_{\nu_{2}})=-\kappa(C)$ for any $C \in \mathcal C_2$;
  \item \emph{Continuity}: Let $C_n \in \mathcal C_2$, $n\in\IN$, and
    $C  \in \mathcal C_2$ with $C_n$ converging
    pointwise to $C$ as $n\rightarrow \infty$. Then $\lim_{n\rightarrow \infty}\kappa(C_n)=\kappa(C).$
 \end{enumerate}
  \end{definition}

  Consider a class of maps on $\mathcal C_2$ written as
  $\kappa_{g_1,g_2}(U,V) = \rho(g_1(U),g_2(V))$, $(U,V)\sim C$, for two
  left-continuous functions $g_1,\ g_2:(0,1)\rightarrow \IR$.
  \cite{koike2022matrix} showed that for $\kappa_{g_1,g_2}$ to be a measure of
  concordance, it must be the so-called $G$-transformed rank correlation defined
  as follows.  For a univariate distribution function
  $G: \IR \rightarrow [0,1]$, the quantile function of $G$ is defined by
\begin{align*}
 G^\i(p)=\inf\{x\in\IR:G(x)\ge p\},\quad p\in(0,1).
\end{align*}
A multivariate distribution $H$ on $\IR^d$ with finite first moment is called radially symmetric if $\bX - \bmu \deq \bmu-\bX$ for some $\bmu \in \IR^d$, where $\bX \sim H$ and $\deq$ stands for equality in distribution.

\begin{definition}[$G$-transformed rank correlation]\label{def:G:transformed:corr}
For a univariate distribution function $G: \IR \rightarrow [0,1]$, the $G$-transformed rank correlation of $C \in \mathcal C_2$ is defined by
\begin{align*}
\kappa_{G}(C)=\rho(G^\i(U),G^\i(V)),\quad (U,V)\sim C.
\end{align*}
We call $G$ concordance-inducing if it is nondegenerate, radially symmetric with finite second moment.
The set of all concordance-inducing distributions is denoted by $\mathcal G$.
\end{definition}
The following proposition summarizes basic properties of $\kappa_G$; see \cite{hofert2019compatibility}.

  \begin{proposition}[Basic properties of $\kappa_G$]\label{prop:properties:g:transformed:correlations}
  For any $G\in \mathcal G$, the $G$-transformed rank correlation $\kappa_G$ satisfies the following properties:
  \begin{enumerate}
\item\label{item:kappa:g:is:moc} $\kappa_G$ is a measure of concordance;
\item\label{item:kappa:g:invariance:location:scale:transform} $\kappa_G$ is invariant under location-scale transforms of $G$, that is, $\kappa_{G_{\mu,\sigma}}(C)=\kappa_G(C)$ for all $C\in \mathcal C_2$, where $\mu\in \IR$, $\sigma>0$ and $G_{\mu,\sigma}(x)=G\left((x-\mu)/\sigma\right)$, $x \in \IR$;
\item\label{item:linearity:kappa:g} For $n \in \IN$, let $C_1,\dots,C_n \in \mathcal C_2$ and
  $\alpha_1,\dots,\alpha_n$ be non-negative numbers such that
  $\alpha_1+\cdots+\alpha_n = 1$. Then
  \begin{align*}
    \kappa_G\biggl(\,\sum_{i=1}^{n}\alpha_i C_i \biggr)=\sum_{i=1}^{n}\alpha_i \kappa_G(C_i).
  \end{align*}
  \end{enumerate}
  \end{proposition}

The class of $G$-transformed rank correlations includes popular measures of concordance as special cases.

\begin{example}[Examples of $G$-transformed rank correlations]
\label{ex:transformed:rank:correlations}
\hspace{2mm}
\begin{enumerate}
\item[1)]\label{example:spearman:rho} Spearman's rho:
$\kappa_G$ reduces to Spearman's rho $\rho_{\text{S}}(U,V)=12\E[UV]-3$ \citep{spearman1904general} if $G$ is the standard uniform distribution $\Unif(0,1)$.
\item[2)]\label{example:blomqvist:beta} Blomqvist's beta:
$\beta(C)=4C(1/2,1/2)-1$ is called Blomqvist's beta \cite{blomqvist1950measure} (also known as median correlation), which is a $G$-transformed rank correlation with $G$ being a symmetric Bernoulli distribution $\Bern(1/2)$ on $\{0,1\}$.
\item[3)]\label{example:van:der:Waerden}
van der Waerden's coefficient:
When $G$ is the standard normal distribution $\N(0,1)$, then $\kappa_G$ is known as van der Waerden's coefficient \cite{sidak1999theory} (also known as normal score correlation and Gaussian rank correlation) $\zeta(U,V)=\rho(\Phi^\i(U),\Phi^\i(V))$  where $\Phi$ is the distribution function of $\N(0,1)$.
\end{enumerate}
\end{example}

\section{Estimation of $\kappa_G$ and their comparison}\label{sec:estimation:kappa:g:comparison}

In this section, we propose a novel approach for comparing $G$-transformed rank correlations to address the question which concordance-inducing distribution is most preferable to use.
In the proposed approach, transformed rank correlations are compared in terms of the asymptotic variances of their canonical estimators, and one concordance-inducing distribution $G\in \mathcal G$ is considered more preferable to another $G'\in \mathcal G$ if the worst and best asymptotic variances of an estimator $\hat \kappa_{G}$ of $\kappa_{G}$ among a set of copulas $\mathcal D \subseteq \mathcal C_2$ are smaller than those of $\kappa_{G'}$.

\subsection{Canonical estimator of $\kappa_G$}\label{subsec:canonical:estimator:kappa:G}

Based on Proposition~\ref{prop:properties:g:transformed:correlations}
Part~\ref{item:kappa:g:invariance:location:scale:transform}, we first consider
standardized concordance-inducing distributions $G$ such that $\E_G[X]=0$ and
$\Var_G(X)=1$ where $X\sim G$.  Suppose that a data-generating
  i.i.d.\ process $(U_1,V_1),(U_2,V_2),\dots \iidsim C$ on the probability space
  $(\Omega,\mathcal A, \Prob)$ is available to estimate a $G$-transformed rank
  correlation $\kappa_{G}$.  This situation corresponds to the case when
  marginal distributions of the i.i.d.\ data are known. Although it may be
  unrealistic, this assumption is imposed throughout the paper to simplify the
  analysis. We then consider the following natural estimator of $\kappa_{G}$.

\begin{definition}[Canonical estimator of $\kappa_G$]\label{def:canonical:estimator:kappa:G}
For $G \in \mathcal G$, the canonical estimator of $\kappa_G$ is given by
\begin{align*}
\hat \kappa_G=\hat \kappa_G^{[n]}(C)=\frac{1}{n}\sum_{i=1}^n G^\i(U_i)\,G^\i(V_i),\quad (U_1,V_1),\dots,(U_n,V_n)\iidsim C, \quad  n \in \IN.
\end{align*}
\end{definition}

Although this simplified setting may be rarely the case, an asymptotic variance of this canonical estimator is obtained in a tractable form.
To this end, let
\begin{align*}
\mathcal G_4=\{G \in \mathcal G: \E_G[X]=0,\quad \Var_G(X)=1,\quad \E_G[X^4]<\infty\text{ for }X \sim G\}.
\end{align*}
If the fourth moment exists, the canonical estimator $\hat \kappa_G$ satisfies
the asymptotic normality by the classical central limit theorem: as $n \rightarrow \infty$,
\begin{align*}
\sqrt{n}\left\{\hat \kappa_G - \kappa_G(C)\right\} \darrow \N(0,\sigma_G^2(C)),
\end{align*}
where the asymptotic variance of $\hat \kappa_G$ is given by
\begin{align*}
\sigma_G^2(C)=\Var(G^\i(U)G^\i(V)),\quad C \in \mathcal C_2.
\end{align*}

\begin{example}[Discrete concordance-inducing distributions]\label{ex:discrete:concordance:inducing:distribution}
For $m\in \IN$, $\bz=(z_1,\dots,z_m) \in \IR^m$ and $\bp=(p_0,p_1,$ $\dots,p_m) \in \IR^{m+1}$ such that $0=z_{0}<z_1<\cdots<z_m$, $p_0 + 2\sum_{i=1}^m p_i=1$ and $\sum_{i=1}^m p_i z_i^2=1/2$, consider a discrete distribution $G_{m,\bz,\bp}$ supported on $-z_m,\dots,-z_1,z_{0},z_1,\dots,z_m$ with corresponding probabilities $p_m,\dots,p_1,p_0,p_1,\dots,p_m\geq 0$.
Then $G_{m,\bz,\bp}$ is a concordance-inducing distribution with mean zero and variance one.
As a special case, Blomqvist's beta arises when $m=1$, $z_1=1$ and $(p_0,p_1)=(0,1/2)$.
Let $p_{+}=p_1+\cdots + p_m$, $I_{-i}=[p_{+}-\sum_{j=1}^{i} p_j,\ p_{+}-\sum_{j=1}^{i-1} p_j]$, $I_0=[p_{+},\ p_{+}+p_0]$ and $I_i=[p_{+}+p_0+\sum_{j=1}^{i-1} p_j,\ p_{+}+p_0+\sum_{j=1}^i p_j]$ for $i\in \{1,\dots,m\}$.
Then
\begin{align*}
\kappa_{G_{m,\bz,\bp}}(C)=\E[G_{m,\bz,\bp}^\i(U)G_{m,\bz,\bp}^\i(V)]=\sum_{(i,j) \in \{-m,\dots,m\}}z_i z_j V_C(I_i\times I_j),
\end{align*}
and
\begin{align*}
\sigma_{G_{m,\bz,\bp}}^2(C) &=\Var(XY)=\E[(XY)^2]-(\E[XY])^2\\
&=  \sum_{(i,j) \in \{-m,\dots,m\}}z_i^2 z_j^2 V_C(I_i\times I_j)- \left(\sum_{(i,j) \in \{-m,\dots,m\}}z_i z_j V_C(I_i\times I_j)\right)^2,
\end{align*}
where $z_{-i}=-z_i$ for $i\in \{1,\dots,m\}$ and $V_C(A)$, $A \subseteq [0,1]^2$, is a volume of $A$ measured by $C$.
\end{example}

For an underlying copula $C$, a concordance-inducing distribution $G$ can be more preferable to another one $G'$ in terms of the stability of statistical estimation if $\sigma_{G}^2(C)\leq \sigma_{G'}^2(C)$.
Since this comparison is valid only for a specific copula $C$, we introduce a set of copulas $\mathcal D$ as possible dependence structures which the analyst is interested in quantifying and comparing.
Concordance-inducing distributions are then compared by the largest and smallest values of $\sigma_G^2(C)$ over $\mathcal D$.

\begin{definition}[Best and worst asymptotic variances for $\kappa_G$]
\label{Def:best:worst:asymptotic:variances}
For $G \in \mathcal G_4$ and $\mathcal D\subseteq \mathcal C_2$, the best and worst asymptotic variances are given by
\begin{align}\label{eq:best:worst:asymptotic:variance}
\underline \sigma_G^2(\mathcal D)=\inf_{C \in \mathcal D} \sigma_G^2(C),\quad
\overline \sigma_G^2(\mathcal D)=\sup_{C \in \mathcal D} \sigma_G^2(C),
\end{align}
respectively.
If the infimum and supremum in~\eqref{eq:best:worst:asymptotic:variance} are attainable, the sets of their attaining copulas on $\mathcal D$  are denoted, respectively, by
\begin{align*}
\underline C_G(\mathcal D)=\argmin_{C \in \mathcal D} \sigma_G^2(C),\quad
\overline C_G(\mathcal D)=\argmax_{C \in \mathcal D} \sigma_G^2(C).
\end{align*}
\end{definition}

Suppose that $\mathcal H \subseteq \mathcal G_4$ represents the set of candidates among which the analyst chooses a concordance-inducing distribution.
In terms of the stability of statistical estimation, we are interested in concordance-inducing distributions that minimize $G\mapsto \underline \sigma_G^2(\mathcal D)$ and/or $G\mapsto \overline \sigma_G^2(\mathcal D)$.

\begin{definition}[Optimal best and worst asymptotic variances]
\label{Def:optimal:best:worst:asymptotic:variances}
For $\mathcal H \subseteq \mathcal G_4$ and $\mathcal D\subseteq \mathcal C_2$, the optimal best and worst asymptotic variances on $(\mathcal H,\mathcal D)$ are defined, respectively, by
\begin{align}\label{eq:optima:best:worst:asymptotic:variance}
\underline\sigma_{\ast}^2(\mathcal H,\mathcal D)=\inf_{G \in \mathcal H}\underline \sigma_G^2(\mathcal D),
\quad
\overline\sigma_{\ast}^2(\mathcal H,\mathcal D)=\inf_{G \in \mathcal H}\overline \sigma_G^2(\mathcal D).
\end{align}
If the infima in~\eqref{eq:optima:best:worst:asymptotic:variance} are attainable, the sets of their attaining concordance-inducing distributions are denoted, respectively, by
\begin{align*}
\underline G_{\ast}(\mathcal H,\mathcal D)=\argmin_{G \in \mathcal H}\underline \sigma_G^2(\mathcal D),
\quad
\overline G_{\ast}(\mathcal H,\mathcal D)=\argmin_{G \in \mathcal H}\overline \sigma_G^2(\mathcal D).
\end{align*}
Finally, the set of optimal concordance-inducing distributions on $(\mathcal H,\mathcal D)$ is given by
\begin{align*}
G_{\ast}(\mathcal H,\mathcal D)=\underline G_{\ast}(\mathcal H,\mathcal D)\cap \overline G_{\ast}(\mathcal H,\mathcal D).
\end{align*}
\end{definition}

Regarding the attainability in~\eqref{eq:best:worst:asymptotic:variance} and~\eqref{eq:optima:best:worst:asymptotic:variance}, we will see in Section~\ref{sec:optimal:concordance:inducing:distributions} that the attaining elements in $\mathcal H$ and $\mathcal D$ can be described explicitly under certain assumptions on $\mathcal H$ and $\mathcal D$.

The comparison of optimal best and worst asymptotic variances leads to the preference order among concordance-inducing distributions as follows.

\begin{definition}[Preference of concordance-inducing distributions]
\label{def:preference:optimal:concordance:inducing:distribution}
We say that $G \in \mathcal G_4$ (or $\kappa_G$) is more preferable to $G' \in \mathcal G_4$ (or $\kappa_{G'}$) on $\mathcal D \subseteq \mathcal C_2$, denoted by $G' \leq_{\mathcal D}G$ (or $\kappa_{G'}\leq_{\mathcal D} \kappa_G$), if
\begin{align*}
\underline \sigma_G^2(\mathcal D)\leq \underline \sigma_{G'}^2(\mathcal D),\quad
\overline \sigma_G^2(\mathcal D)\leq \overline \sigma_{G'}^2(\mathcal D).
\end{align*}
\end{definition}

By definition, the preference order $G' \leq_{\mathcal D}G$ is a partial order except some restricted cases of $\mathcal D$ as seen in Corollary~\ref{cor:optimal:concordance:inducing:functions:fundamental:frechet:copulas} below.
If $G \in G_{\ast}(\mathcal H,\mathcal D)$, then $G$ may be considered as the most preferable choice among $\mathcal H\subseteq \mathcal G_4$ to accurately estimate $\kappa_G$ if the analyst believes that $\mathcal D$ is the set of underlying copulas on which she wants to quantify and compare dependence.

Other than $\mathcal H = \mathcal G_4$, one may be interested, for example, in $\mathcal H=\mathcal G_4^{\text{c}}$ where $\mathcal G_4^{\text{c}}$ is the set of continuous concordance-inducing distributions in $\mathcal G_4$, and in $\mathcal H=\mathcal G_4^{\text{b}}$ where $\mathcal G_4^{\text{b}}$ is the set of concordance-inducing distributions in $\mathcal G_4$ with bounded supports.
Note that one-sided distributions $X \sim G$ such that $\esssup(X)=\infty$ and $\essinf(X)<\infty$, or $\esssup(X)<\infty$ and $\essinf(X)=-\infty$, cannot be concordance-inducing since they cannot be radially symmetric.
Therefore, $\mathcal G_4\backslash \mathcal G_4^{\text{b}}$ is a set of concordance-inducing distributions supported on $\IR$.

\subsection{Properties of the asymptotic variance}
\label{subsec:properties:asymptotic:variance}

We study properties of the asymptotic variance and its optimal best and worst counterparts.

\begin{proposition}[Basic properties of $\sigma_G^2(C)$]
\label{prop:basic:properties:asymptotic:variance}
Let $G \in \mathcal G_4$.
Then the map $C \mapsto \sigma_G(C)$ satisfies the following properties:
\begin{enumerate}
\item\label{item:covariance:formula:asymptotic:variance}
\emph{Covariance formula}:
For $(U,V)\sim C \in \mathcal C_2$ and $(X,Y)=(G^\i(U),G^\i(V))$, we have that
\begin{align}\label{eq:covariance:formula:asymptotic:variance}
\sigma_G^2(C)=\Cov(X^2,Y^2)+1 - \Cov(X,Y)^2;
\end{align}
\item\label{item:fundamental:copulas:asymptotic:variance}
\emph{Values at fundamental copulas}:
$\sigma_G^2(\Pi)=1$ and $\sigma_G^2(M)=\sigma_G^2(W)=\Var_G(X^2)$;
\item\label{item:bounds:asymptotic:variance}
\emph{Bounds}: $0\leq \sigma_G^2(C)\leq 1+\Var_G(X^2)$ for $X \sim G$.
In particular, $\sigma_G^2\left((M+W)/2\right)= 1+\Var_G(X^2)$;
\item\label{item:boundedness:asymptotic:variance}
\emph{Boundedness}:
$\sigma_G^2(C)<\infty$ for all $C \in \mathcal C_2$;
\item\label{item:reflection:invariance:asymptotic:variance}
\emph{Reflection invariance}:
$\sigma_G^2(C)=\sigma_G^2(C_{\nu_1})=\sigma_G^2(C_{\nu_2})=\sigma_G^2(C_{\nu_1\circ \nu_2})$.
\end{enumerate}
\end{proposition}

\begin{proposition}[Concavity of $\sigma_G^2(C)$]\label{prop:concavity:sigma:G:C:asymptotic:variance}
For $G \in \mathcal G_4$, the map $C \mapsto \sigma_G^2(C)$ is concave with respect to convex combinations of copulas.
In particular, for $k \in [-1,1]$, the map $C \mapsto \sigma_G^2(C)$ is linear on $\mathcal C_G(k)=\{C \in \mathcal C_2: \kappa_G(C)=k\}$.
\end{proposition}

Concavity of $C \mapsto \sigma_G^2(C)$ can be helpful to simplify the problem of deriving $\underline \sigma_G^2(\mathcal D)$ in some cases.

\begin{corollary}[$\underline \sigma_G^2(C)$ over convex combinations of copulas]\label{cor:min:lower:asymptotic:variance:extremal:points}
For $L \in \IN$ and $C_1,\dots,C_L \in \mathcal C_2$, let 
$$\mathcal D=\operatorname{conv}(C_1,\dots,C_L) \subseteq \mathcal C_2$$
be the set of convex combinations of $C_1,\dots,C_L$.
Then
$$
\underline \sigma_G^2(\mathcal D)= \min_{l=1,\dots,L}\sigma_G^2(C_l),\quad
\underline C_G(\mathcal D)\supseteq \argmin_{l=1,\dots,L}\sigma_G^2(C_l).
$$
\end{corollary}

\subsection{Optimal location shift of $G$}\label{subsec:optimal:location:shift}

Although $\kappa_G$ is invariant under location-scale transforms of $G$ by Proposition~\ref{prop:properties:g:transformed:correlations} Part~\ref{item:kappa:g:invariance:location:scale:transform}, the asymptotic variance $\sigma_G^2(C)$ of its canonical estimator $\hat \kappa_G$ may not be the case.
To see this, let $G_0 \in \mathcal G_4$ be a standardized concordance-inducing distribution with mean zero and variance one, and let $G_{\mu,\sigma}(x)=G_0((x-\mu)/\sigma)$ be the corresponding concordance-inducing distribution of the same type as $G_0$ but with mean $\mu \in \IR$ and variance $\sigma^2 >0$.
Since
$$
\kappa_{G_{\mu,\sigma}}(C)=\rho(G_{\mu,\sigma}^\i(U),G_{\mu,\sigma}^{\i}(V))=\frac{\E\left[G_{\mu,\sigma}^\i(U)G_{\mu,\sigma}^\i(V)\right]- \mu^{2}}{\sigma^{2}},\quad (U,V)\sim C,
$$
for known $\mu$ and $\sigma$, a canonical estimator of $\kappa_{G_{\mu,\sigma}}$ can be given by
\begin{align*}
\hat \kappa_{G_{\mu,\sigma}}=\hat \kappa_{G_{\mu,\sigma}}^{[n]}(C) = \frac{1}{n}\sum_{i=1}^n \frac{G_{\mu,\sigma}^\i(U_i)G_{\mu,\sigma}^\i(V_i)}{\sigma^2} - \left(\frac{\mu}{\sigma}\right)^2,
\end{align*}
for $(U_{1},V_{1}),(U_{2},V_{2}),\dots \iidsim C$, which reduces to~\eqref{def:canonical:estimator:kappa:G} when $\mu=0$ and $\sigma^{2}=1$.
By the CLT, asymptotic normality follows for $\hat \kappa_{G_{\mu,\sigma}}$ with the asymptotic variance given by
\begin{align*}
\sigma_{G_{\mu,\sigma}}^2(C)=\Var\left(\frac{G_{\mu,\sigma}^\i(U)G_{\mu,\sigma}^\i(V)}{\sigma^2}\right).
\end{align*}
Since $G_{\mu,\sigma}^\i(U)/\sigma=G_{\mu/\sigma,1}^\i(U)$ and $G_{\mu,\sigma}^\i(V)/\sigma=G_{\mu/\sigma,1}^\i(V)$, one can assume that $\sigma=1$ without changing the asymptotic variance $\sigma_{G_{\mu,\sigma}}^2(C)$, that is, $\sigma_G^2(C)$ is invariant under scale transforms of $G$.
On the other hand, $\sigma_G^2(C)$ changes under location transforms of $G$ since shifting $G^\i$ by $\mu \in \IR$ leads to the asymptotic variance $\Var((X+\mu)(Y+\mu))=\Var(XY+\mu(X+Y))$ for $X=G^\i(U)$ and $Y=G^\i(V)$, which is in general not equal to $\Var(XY)$.

Since the canonical estimator $\hat \kappa_{G_{\mu,\sigma}}$ estimates the same quantity $\kappa_{G_0}$ regardless of the mean $\mu$ and variance $\sigma^2$ of $G$, a natural choice of $(\mu,\sigma)$ is such that $\sigma =1$ and $\mu$ minimizes the asymptotic variance $\sigma_{G_{\mu,1}}^2(C)$.
For a fixed concordance-inducing distribution $G_0 \in \mathcal G_4$ with mean zero and variance one, denote by $G_\mu(x)=G_0(x-\mu)$ the concordance-inducing distribution of the same type as $G_0$ but with mean $\mu\in \IR$.
For $X= X_0+\mu \sim G_\mu$ and $Y= Y_0+\mu \sim G_\mu$ with $X_0=G_0^\i(U)$ and $Y_0=G_0^\i(V)$, the asymptotic variance is given by
\begin{align*}
\sigma_{G_\mu}^2(C)&=\Var(XY)=\Var((X_0+\mu)(Y_0+\mu))=\Var(X_0Y_0+\mu(X_0+Y_0))\\
&= \Var(X_0Y_0)+ 2\mu \Cov(X_0Y_0,X_0+Y_0)+\mu^2 \Var(X_0+Y_0).
\end{align*}
Therefore the desired $\mu$ is given as follows.

\begin{definition}[Optimal shift of $G_0$]
For $G_0 \in \mathcal G_4$ and $C \in \mathcal C_2$, the minimizer of $\mu \mapsto \sigma_{G_\mu}^2(C)$ is called the optimal (location) shift of $G_0$ under $C$, and is given by
\begin{align*}
\mu_{\ast} = \mu_{\ast}(G_0,C)=\begin{cases}
-\frac{\Cov(X_0Y_0,\, X_0+Y_0)}{\Var(X_0+Y_0)}, & \text{ if } \Var(X_0+Y_0)>0,\\
0, & \text{ if } \Var(X_0+Y_0)=0,
\end{cases}
\end{align*}
where $(U,V)\sim C$ and $(X_0,Y_0)=(G_0^\i(U),G_0^\i(V))$.
The optimal asymptotic variance is then given by
\begin{align*}
\sigma_{G_{\mu_\ast}}^2(C)=
\begin{cases}
\Var(X_0Y_0)-\frac{\Cov(X_0Y_0,X_0+Y_0)^2}{\Var(X_0+Y_0)},& \text{ if } \Var(X_0+Y_0)>0,\\
\Var(X_0Y_0), & \text{ if } \Var(X_0+Y_0)=0.
\end{cases}
\end{align*}

\end{definition}

For $G_0 \in \mathcal G_4$, the degenerate case $\Var(X_0+Y_0)=0$ occurs if and only if $X_0+Y_0\aseq0$.
In this case, we have that
\begin{align*}
\sigma_{G_\mu}^2(C)= \Var(X_0Y_0)+ 2\mu \Cov(X_0Y_0,0)+\mu^2 \Var(0)= \Var(X_0Y_0)=\sigma_{G_0}^2(C),
\end{align*}
for every $\mu \in \IR$.

The following proposition states that $\mu_{\ast}=0$ for a certain class of copulas.

\begin{proposition}[Sufficient condition for $\mu_{\ast}=0$]\label{prop:sufficient:condition:optimal:shift:zero}
Let $C \in \mathcal C_2$ be a copula and $G_0 \in \mathcal G_4$ be a concordance-inducing distribution with mean zero and variance one.
Then $\mu_{\ast}(G_0,C)=0$ holds if $C$ is radially symmetric $C=C_{\nu_1\circ \nu_2}$, that is,
$(U,V)\deq (1-U,1-V)$ for $(U,V)\sim C$.
\end{proposition}

By Proposition~\ref{prop:sufficient:condition:optimal:shift:zero}, a location shift of $G_0$ does not change the asymptotic variance $\sigma_{G_0}^2(C)$ when $C$ is, for example, $M$, $W$, $\Pi$, a Gaussian copula, $t$ copula or their mixtures.
On the other hand, shifting $G_0$ may improve $\sigma_{G_0}^2(C)$ if $C$ is, for example, a Clayton or Gumbel copula since they are not radially symmetric in general.
The next proposition states that the optimal asymptotic variance can be obtained analytically when $(X,Y)=(G^\i(U),G^\i(V))$ is a normal variance mixture, that is,
\begin{align}\label{eq:normal:variance:mixture}
(X,Y)=(\mu_1,\mu_2)+\sqrt{W}(Z_1,Z_2),
\quad (\mu_1,\mu_2)\in \IR^2,
\quad (Z_1,Z_2)\sim \N_2(\bzero_2,\Sigma),
\end{align}
where $\Sigma$ is a $2$-dimensional square positive definite matrix and $W \geq 0$ is a non-negative random variable independent of $(Z_1,Z_2)$ and such that $\Prob(W=0)<1$.

\begin{proposition}[$\sigma_{G_{\mu_\ast}}^2(C)$ for normal variance mixture]\label{prop:asymptotic:variance:normal:variance:mixture}
For $G \in \mathcal G_4$ and $(U,V)\sim C \in \mathcal C_2$, suppose that $(X,Y)=(G^\i(U),G^\i(V))$ is a normal variance mixture specified by~\eqref{eq:normal:variance:mixture}.
Then the optimal asymptotic variance is given by
\begin{align}\label{eq:optimal:asymptotic:variance:normal:variance:mixture}
\sigma_{G_{\mu^\ast}}^2(C)= \left(2\frac{\E[W^2]}{\E[W]^2}  -1\right) \rho^2(X,Y) + \frac{\E[W^2]}{\E[W]^2}.
\end{align}
\end{proposition}

By~\eqref{eq:optimal:asymptotic:variance:normal:variance:mixture}, the function $\rho \mapsto \sigma_{G_{\mu_\ast}}^2(C)$ is convex since $\E[W^2]\geq \E[W]^2$.
For normal distributions, $W=1$ and thus $\sigma_{G_{\mu_\ast}}^2(C)=\rho^2(X,Y)+1$.
Finally, as we will see in Section~\ref{sec:simulation:study}, the variance $\Var(X^2)$, $X\sim G$, is observed to affect convexity or concavity of $\rho \mapsto \sigma_{G_{\mu_\ast}}^2(C)$.
For the case of normal variance mixtures,
we have that
\begin{align}
\label{eq:optimal:asymptotic:variance:normal:variance:mixture:variance:X:squared}
\sigma_{G_{\mu^\ast}}^2(C)= \left(\Var(X^2)-\frac{\E[W^2]}{\E[W]^2}\right) \rho^2(X,Y) + \frac{\E[W^2]}{\E[W]^2},
\end{align}
since
\begin{align*}
\Var(X^2)&=\E[W^2Z_1^4]-(\E[WZ_1^2])^2
=\E[W^2]\left(\frac{3}{\E[W]^2}\right)-\E[W]^2\left(\frac{1}{\E[W]^2}\right)
=3\frac{\E[W^2]}{\E[W]^2}-1.
\end{align*}

\section{Optimal concordance-inducing distributions}
\label{sec:optimal:concordance:inducing:distributions}

In this section we investigate optimal best and worst asymptotic variances and their attaining concordance-inducing distributions for certain choices of $\mathcal D\subseteq \mathcal C_2$.

\subsection{Asymptotic variance for fundamental and Fr\'echet copulas}\label{subsec:asymptotic:variance:fundamental:copulas}

We first consider the case when $\mathcal D \subset \mathcal C_2$ is a set of fundamental copulas $M$, $\Pi$ and $W$, or their mixtures since these copulas play important roles in the discussion of the best and worst asymptotic variances.
By radial symmetry of these copulas, the optimal shift $\mu_\ast$ is zero and thus it suffices to consider standardized concordance-inducing distributions in $\mathcal G_4$.

\begin{definition}[Fr\'echet copula]\label{def:frechet:copulas}
A bivariate Fr\'echet copula is defined by
\begin{align*}
C_{\bp}^{\operatorname{F}} =  p_M M + p_{\Pi}\Pi + p_{W} W,\quad  \bp = (p_M,p_{\Pi},p_W)\in \Delta_3,
\end{align*}
where $\Delta_3=\{(p_1,p_2,p_3)\in \IR^3: p_{1},\,p_{2},\,p_{3}\geq 0,\,p_1+p_2+p_3=1\}$ is the standard unit simplex on $\IR^3$.
The set of all Fr\'echet copulas is denoted by $\mathcal C^{\operatorname{F}}= \{C_{\bp}^{\operatorname{F}}: \bp \in \Delta_3\}$.
\end{definition}

In addition to their financial applications, Fr\'echet copulas can be used to approximate bivariate copulas; see \cite{yang2006bivariate}.
Moreover, for any $G \in \mathcal G_4$, the transformed rank correlation $\kappa_G$ can take any value in $[-1,1]$ since, by Proposition~\ref{prop:properties:g:transformed:correlations} Part~\ref{item:linearity:kappa:g}, it holds that
\begin{align}\label{eq:kappa:g:Frechet:copulas}
\kappa_G(C_{\bp}^{\operatorname{F}})=p_M \kappa_G(M) + p_{\Pi}\kappa_G(\Pi) + p_{W} \kappa_G(W)=p_M-p_W\in [-1,1].
\end{align}

The following proposition is an immediate consequence from Proposition~\ref{prop:basic:properties:asymptotic:variance} Part~\ref{item:fundamental:copulas:asymptotic:variance}.

\begin{proposition}[Optimal asymptotic variances for fundamental copulas]\label{prop:optimal:g:asymptotic:variance:fundamental:copulas}
Let $\mathcal H\subseteq \mathcal G_4$.
\begin{enumerate}
  \item\label{item:optimal:G:independence}
  $\underline\sigma_{\ast}^2(\mathcal H,\{\Pi\})=\overline\sigma_{\ast}^2(\mathcal H,\{\Pi\})=1$
   and
   $\underline G_\ast(\mathcal H,\{\Pi\})=\overline G_\ast(\mathcal H,\{\Pi\})=\mathcal H$.
 \item\label{item:optimal:G:extremal:dependence}
   Suppose $\mathcal D=\{M\}$, $\{W\}$ or $\{M,W\}$. Then
  \begin{align*}
      \underline \sigma_\ast^2(\mathcal H,\mathcal D)=\overline \sigma_\ast^2(\mathcal H,\mathcal D)=\inf_{G \in \mathcal H}\Var_G(X^2).
      \end{align*}
If this infimum is attainable, then
\begin{align*}
       \underline G_\ast(\mathcal H,\mathcal D)=\overline G_\ast(\mathcal H,\mathcal D)=\argmin_{G \in \mathcal H}\Var_G(X^2).
  \end{align*}
  \item\label{item:optimal:G:fundamental:copulas}
  Suppose that $\mathcal D=\{\Pi,M,W\}$. Then
\begin{align*}
\underline \sigma_\ast^2(\mathcal H,\{M,\Pi,W\})= 1 \wedge \inf_{G \in \mathcal H}\Var_G(X^2),\quad
\overline \sigma_\ast^2(\mathcal H,\{M,\Pi,W\})=1 \vee \inf_{G \in \mathcal H}\Var_G(X^2).
\end{align*}
If the infima above are attainable, then
\begin{align*}
\underline G_{\ast}(\mathcal H, \{M,\Pi,W\})&=
\begin{cases}
\argmin_{G \in \mathcal H}\Var_G(X^2), &  \text{ if } \min_{G \in \mathcal H}\Var_G(X^2)<1,\\
\mathcal H, & \text{ if }\min_{G \in \mathcal H}\Var_G(X^2)\geq 1,
\end{cases}\\
\overline G_{\ast}(\mathcal H, \{M,\Pi,W\})&=
\begin{cases}
\mathcal H, &  \text{ if } \min_{G \in \mathcal H}\Var_G(X^2)<1,\\
\argmin_{G \in \mathcal H}\Var_G(X^2), & \text{ if }\min_{G \in \mathcal H}\Var_G(X^2)\geq 1.
\end{cases}
\end{align*}
\end{enumerate}
\end{proposition}

The next proposition provides the best and worst asymptotic variances and their attainers when $\mathcal D=\mathcal C^{\operatorname{F}}$.

\begin{proposition}[Best and worst asymptotic variances for Fr\'echet copulas]\label{prop:worst:best:asymptotic:variances:Frechet:copulas}
For a concordance-inducing distribution $G \in \mathcal G_4$, the best and worst asymptotic variances on $\mathcal C^{\operatorname{F}}$ are given by
\begin{align*}
\underline \sigma_{G}^2(\mathcal C^{\operatorname{F}})=1 \wedge \Var_G(X^2),
\quad
\overline \sigma_{G}^2(\mathcal C^{\operatorname{F}})=1+\Var_G(X^2)
\end{align*}
with the sets of attaining copulas given by
\begin{align*}
\underline C_G(\mathcal \mathcal C^{\operatorname{F}})
&=
\begin{cases}
\{M, W\}, & \text{ if } 0\leq \Var_G(X^2)<1,\\
\{M, \Pi,W\},& \text{ if }\Var_G(X^2)=1,\\
\{\Pi\},& \text{ if }\Var_G(X^2)>1,
\end{cases}\\
\overline C_G(\mathcal C^{\operatorname{F}})
&=
\begin{cases}
\left\{ \frac{M+W}{2} \right\}, & \text{ if } \Var_G(X^2)>0,\\
\left\{p \frac{M+W}{2} + (1-p)\Pi: p \in [0,1] \right\}, & \text{ if } \Var_G(X^2)=0.
\end{cases}
\end{align*}
\end{proposition}

Note that the result in Proposition~\ref{prop:worst:best:asymptotic:variances:Frechet:copulas} is consistent with Corollary~\ref{cor:min:lower:asymptotic:variance:extremal:points}.
In the proof of Proposition~\ref{prop:worst:best:asymptotic:variances:Frechet:copulas}, although $(p_M,p_W)=(1/2,1/2)$ is the unique point attaining the maximum $v+1$ of $f$ when $v>0$, $f$ takes the value $v$ at the points $(p_M,p_W)=(1,0)$ and $(0,1)$, and is greater than $v$ on $\{(p_M,p_W) \in [0,1]^2: p_M+p_W=1\}$.
Therefore, if $v=\Var_G(X^2)$ is sufficiently large, the asymptotic variance $\sigma_G^2(C)$
takes large values in $[\Var_G(X^2),\Var_G(X^2)+1]$ if $C=p M + (1-p)W$ for $p \in [0,1]$.

\begin{remark}[Restrictions of $\mathcal C^{\operatorname{F}}$]
For a concordance-inducing distribution $G \in \mathcal G_4$, consider the set of Fr\'echet copulas such that its transformed rank correlation $\kappa_G$ takes values in $[\underline k,\overline k]$ for $-1 \leq \underline k \leq \overline k \leq 1$, that is,
\begin{align*}
\mathcal C_{\underline k,\overline k}^{\operatorname{F}}(G)=\{C \in \mathcal C^{\operatorname{F}}: \underline k \leq \kappa_G(C) \leq \overline k\}.
\end{align*}
By~\eqref{eq:kappa:g:Frechet:copulas}, the restriction $\underline k \leq \kappa_G(C) \leq \overline k$ reduces to $\underline k \leq p_M-p_W \leq \overline k$ and thus $\mathcal C_{\underline k,\overline k}^{\operatorname{F}}(G)$ does not depend on the choice of $G$.
Consequently, the maximum and minimum of the asymptotic variance $\sigma_G^2(C)$ on $\mathcal C_{\underline k,\overline k}^{\operatorname{F}}(G)$ can be found by calculating $\max f(p_M,p_W)$ and $\min f(p_M,p_W)$ subject to the constraints
\begin{align*}
\{(p_{M},p_{W}) \in \IR^{2}:0\leq p_M,\,p_W,\, p_M+p_W\leq 1\,\text{ and } \underline k \leq p_M-p_W \leq \overline k\}.
\end{align*}
This maximum and minimum always exist since $(p_M,p_W)\mapsto f(p_M,p_W)$ is bounded, concave and the feasible set is compact in $\IR^2$.
\end{remark}

Proposition~\ref{prop:worst:best:asymptotic:variances:Frechet:copulas} immediately leads to the optimal best and worst asymptotic variances on $\mathcal D=\mathcal C^{\operatorname{F}}$ as stated in the following corollary.

\begin{corollary}[Optimal best and worst asymptotic variances for Fr\'echet copulas]\label{cor:optimal:asymptotic:variances:Frechet:copulas}
For $\mathcal H \subseteq \mathcal G_4$, the optimal best and worst asymptotic variances are given by
\begin{align*}
\underline \sigma_{\ast}^2(\mathcal H, \mathcal C^{\operatorname{F}})=1 \wedge \inf_{G \in \mathcal H}\Var_G(X^2),
\quad
\overline \sigma_{\ast}^2(\mathcal H, \mathcal C^{\operatorname{F}})=1+\inf_{G \in \mathcal H}\Var_G(X^2).
\end{align*}
If the infima above are attainable, then the sets of attaining concordance-inducing distributions are given, respectively, by
\begin{align*}
\underline G_{\ast}(\mathcal H, \mathcal C^{\operatorname{F}})&=
\begin{cases}
\argmin_{G \in \mathcal H}\Var_G(X^2), &  \text{ if } \min_{G \in \mathcal H}\Var_G(X^2)<1,\\
\mathcal H, & \text{ if }\min_{G \in \mathcal H}\Var_G(X^2)\geq 1,
\end{cases}\\
\overline G_{\ast}(\mathcal H, \mathcal C^{\operatorname{F}})&=\argmin_{G \in \mathcal H}\Var_G(X^2).
\end{align*}
\end{corollary}

Compared with the optimal best and worst asymptotic variances from Proposition~\ref{prop:optimal:g:asymptotic:variance:fundamental:copulas} Part~\ref{item:optimal:G:fundamental:copulas}, the lower bound $\underline \sigma_{\ast}^2(\mathcal H,\mathcal D)$ obtained in Proposition~\ref{cor:optimal:asymptotic:variances:Frechet:copulas} remains unchanged whereas the upper bound $\overline \sigma_{\ast}^2(\mathcal H,\mathcal D)$ increases since the attaining copulas $p\left(M+W)\right)/2+(1-p)\Pi$, $p \in [0,1]$, are not included in the set $\mathcal D$ in Proposition~\ref{prop:optimal:g:asymptotic:variance:fundamental:copulas}.
Nevertheless, the best and worst asymptotic variances are the functions of $\Var_G(X^2)$ when $\mathcal D$ is a set of fundamental or Fr\'echet copulas, and thus we have the following result.

\begin{corollary}[Optimal concordance-inducing distributions on fundamental or Fr\'echet copulas]
\label{cor:optimal:concordance:inducing:functions:fundamental:frechet:copulas}
Suppose that $\mathcal D$ is a set of fundamental or Fr\'echet copulas, that is, $\mathcal D=\{\Pi\}$, $\{M\}$, $\{W\}$, $\{M,W\}$, $\{M,\Pi,W\}$ or $\mathcal C^{\text{F}}$. Then
\begin{enumerate}
\item\label{item:total:order:funndamental:frechet:copuas} $\leq_{\mathcal D}$ is a total order;
\item\label{item:preference:order:funndamental:frechet:copuas} $G'\leq_{\mathcal D} G$ if $\Var_{G}(X^2)\leq \Var_{G'}(X^2)$;
\item\label{item:optimal:concordance:inducing:distributions:funndamental:frechet:copuas} $G_{\ast}(\mathcal H, \mathcal D)=\argmin_{G \in \mathcal H}\Var_G(X^2)$ provided that $\inf_{G \in \mathcal H}\Var_G(X^2)$ is attainable.
\end{enumerate}
\end{corollary}

Corollary~\ref{cor:optimal:concordance:inducing:functions:fundamental:frechet:copulas} states that concordance-inducing distributions having a smaller variance of $X^2$ for $X\sim G$ are more preferable in terms of best and worst asymptotic variances when $\mathcal D$ is a set of fundamental or Fr\'echet copulas.
As a consequence, heavy-tailed concordance-inducing distributions, such as a Student $t$ distribution with degrees of freedom $4 < \nu <\infty$, are not recommendable choices at least in terms of accuracy of statistical estimation.
In particular, popular measures of concordance introduced in Example~\ref{ex:transformed:rank:correlations} are ordered as follows.

\begin{corollary}[Preference orders for $\rho_{\text{S}}$, $\beta$ and $\zeta$]
\label{cor:preference;orders:popular:mocs}
Suppose that $\mathcal D$ is a set of fundamental or Fr\'echet copulas.
Then $\zeta \leq_{\mathcal D} \rho_{\text{S}}\leq_{\mathcal D} \beta$.
\end{corollary}

\subsection{Optimality of Blomqvist's beta}\label{subsec:optimality:Blomqvist}

In this section, we show that Blomqvist's beta is an optimal $G$-transformed rank correlation under some conditions on $\mathcal D \in \mathcal C_2$.
For $G \in \mathcal G_4$, we identify $G \in \mathcal G_4$ with $\kappa_G$, and thus we allow $\sigma_{G}^2$ to be written as $\sigma_\beta^2$ for the standardized symmetric Bernoulli distribution $G \in \mathcal G_4$.

\begin{definition}[Balancedness of copulas]\label{def:balanced:copulas}
Let
$$p(C)=C(1/2,1/2)+\bar C(1/2,1/2),\quad C \in \mathcal C_2.
$$
A copula $C \in \mathcal C_2$ is called
\begin{enumerate}
\item[(i)] balanced if $p(C)=1/2$,
\item[(ii)] imbalanced if $p(C)\neq 1/2$,
\item[(iii)] totally positively imbalanced (TPI) if $p(C)=1$,
\item[(iv)]totally negatively imbalanced (TNI) if $p(C)=0$.
\end{enumerate}
\end{definition}

It is straightforward to check that $\Pi$ and $(M+W)/2$ are balanced, $M$ is TPI and $W$ is TNI.

\begin{proposition}[Asymptotic variance of Blomqvist's beta]
\label{prop:asymptotic:variances:beta}
Let $C \in \mathcal C_2$.
Then the following properties hold for $\beta$.
\begin{enumerate}
\item\label{item:optimal:shift:asymptotic:variances:beta} $\mu_{\ast}(\beta,C)=0$;
\item\label{item:explicit:form:asymptotic:variances:beta} $\sigma_\beta^2(C)=4p(C)(1-p(C))=1-\beta^2(C)$;
\item\label{item:upper:lower:bounds:asymptotic:variances:beta} $0\leq \sigma_\beta^2(C)\leq 1$;
\item\label{item:lower:bound:asymptotic:variances:beta} $\sigma_\beta^2(C)=0$ if and only if $C$ is a TPI or TNI copula;
\item\label{item:upper:bound:asymptotic:variances:beta} $\sigma_\beta^2(C)=1$ if and only if $C$ is balanced.
\end{enumerate}
\end{proposition}

\begin{remark}[Asymptotic variance of $\beta$ for elliptical copulas]
\label{remark:asymptotic:variances:Blomqvist:beta:elliptical}
Blomqvist's beta admits an explicit form $\beta(C)=\frac{2}{\pi}\arcsin (\rho)$ when $C$ is an elliptical copula with correlation parameter $\rho\in[-1,1]$.
Therefore, by Proposition~\ref{prop:asymptotic:variances:beta} Part~\ref{item:explicit:form:asymptotic:variances:beta},
we have that $\sigma_{\beta}^2(C)=1-\left(\frac{2}{\pi}\arcsin (\rho)\right)^2$,
which coincides with the result derived in \cite[Proposition~9]{schmid2007nonparametric}.
\end{remark}

Next we prove the optimality of Blomqvist's beta under certain conditions on $\mathcal D$.

\begin{proposition}[Optimality of Blomqvist's beta]\label{prop:optimality:blomqvist:beta}
Let $\mathcal D \subseteq \mathcal C_2$ and $\beta \in \mathcal H\subseteq \mathcal G_4$.
\begin{enumerate}
\item\label{item:optimality:blomqvist:beta:lower}
If $C^\ast \in \mathcal D$ for some TPI or TNI copula $C_\ast$, then $\beta \in \underline G_\ast(\mathcal H,\mathcal D)$.
\item\label{item:optimality:blomqvist:beta:upper}
If $\Pi \in \mathcal D$, then $\beta \in \overline G_\ast(\mathcal H,\mathcal D)$.
\item\label{item:optimality:blomqvist:beta}
If $\Pi,\ C_\ast \in \mathcal D$ for some TPI or TNI copula $C_\ast$, then $\beta \in G_\ast(\mathcal H,\mathcal D)$.
\end{enumerate}
Statements \ref{item:optimality:blomqvist:beta:lower},~\ref{item:optimality:blomqvist:beta:upper} and~\ref{item:optimality:blomqvist:beta} remain valid if $\sigma_G^2(C)$, $G \in \mathcal G_4$, in Definitions~\ref{Def:best:worst:asymptotic:variances},~\ref{Def:optimal:best:worst:asymptotic:variances} and~\ref{def:preference:optimal:concordance:inducing:distribution} is replaced by the optimally shifted asymptotic variance $\sigma_{G_{\mu_{\ast}}}^2(C)$.
\end{proposition}

Proposition~\ref{prop:optimality:blomqvist:beta} states that Blomqvist's beta is an optimal choice of $G$-transformed rank correlation for possibly typical choices of $\mathcal D$, such as $\mathcal C_2$, $\mathcal C_2^{\succeq}= \{C \in \mathcal C_2: C \succeq \Pi\}$ or $\mathcal C_2^{\preceq}= \{C \in \mathcal C_2: C \preceq \Pi\}$.

\subsection{Uniqueness of the optimality of $\beta$}\label{subsec:uniqueness:optimality:beta}

In this section we investigate whether Blomqvist's beta is the unique optimal $G$-transformed rank correlation, that is, whether $G_{\ast}(\mathcal H,\mathcal D)=\{\beta\}$.
The next proposition states that this uniqueness holds under some condition on $\mathcal D$.

\begin{proposition}[Uniqueness of $\beta$ for $\overline G_{\ast}(\mathcal H,\mathcal D)$ and $G_{\ast}(\mathcal H,\mathcal D)$]
\label{prop:uniqueness:beta:worst:G}
Let $\mathcal H \subseteq \mathcal G_4$ and $\mathcal D \subseteq \mathcal C_2$ be such that
\begin{enumerate}
\item $\beta \in \mathcal H$;
\item $\Pi,\,(M+W)/2 \in \mathcal D$ and $\mathcal D$ contains a TPI or TNI copula.
\end{enumerate}
Then $\overline G_{\ast}(\mathcal H,\mathcal D)=\{\beta\}$ and thus $G_{\ast}(\mathcal H,\mathcal D)=\{\beta\}$.
The statement remains valid if $\sigma_G^2(C)$, $G \in \mathcal G_4$, in Definitions~\ref{Def:best:worst:asymptotic:variances},~\ref{Def:optimal:best:worst:asymptotic:variances} and~\ref{def:preference:optimal:concordance:inducing:distribution} is replaced by the optimally shifted asymptotic variance $\sigma_{G_{\mu_{\ast}}}^2(C)$.
\end{proposition}

Proposition~\ref{prop:uniqueness:beta:worst:G} does not address whether $\underline G_\ast(\mathcal H,\mathcal D)=\{\beta\}$.
This, however, is rarely the case as we will see in what follows.

For given $\mathcal H \subseteq \mathcal G_4$ and $\mathcal D \subseteq \mathcal C_2$, assume that $\beta \in \mathcal H$ and that $\mathcal D$ contains at least one TPI or TNI copula.
Then the following equivalence relations hold by Proposition~\ref{prop:asymptotic:variances:beta}~Part~\ref{item:lower:bound:asymptotic:variances:beta}:
\begin{align}\label{eq:equivalent:conditions:attaining:lower:bound:asymptotic:variance:bernoulli}
\nonumber G \in \underline G_{\ast}(\mathcal H,\mathcal D) \quad &\Leftrightarrow \quad
\text{there exists } C \in \mathcal D \text{ such that } \sigma_G^2(C)=0\\
& \Leftrightarrow \quad G^\i(U)G^\i(V)\aseq a \text{ for some } a \in \IR \text{ and } (U,V)\sim C \in \mathcal D.
\end{align}
The next proposition provides necessary conditions on $a \in \IR$, $G \in \mathcal H$ and $C \in \mathcal D$ in~\eqref{eq:equivalent:conditions:attaining:lower:bound:asymptotic:variance:bernoulli}.

\begin{proposition}[Necessary conditions on  $G \in \underline G_{\ast}(\mathcal H,\mathcal D)$]\label{prop:necessary:conditions:g:attain:optimal:best:asymptotic:variance}
For $\mathcal H \subseteq \mathcal G_4$ and $\mathcal D \subseteq \mathcal C_2$, suppose that $\beta \in \mathcal H$ and that $\mathcal D$ contains at least one TPI or TNI copula.
If $G \in \underline G_{\ast}(\mathcal H,\mathcal D)$, then $C \in \mathcal D$ and $a \in \IR$ in~\eqref{eq:equivalent:conditions:attaining:lower:bound:asymptotic:variance:bernoulli} satisfy the following conditions:
\begin{enumerate}
\item\label{item:necessary:conditions:point:mass:case} If $\Prob(X=0)>0$ for $X\sim G$, then $a=0$ and $\Prob(X=0)\geq 1/2$;
\item\label{item:necessary:conditions:no:point:mass:case} If $\Prob(X=0)=0$, then $a \neq 0$ and the copula $C$ is either TPI or TNI with $0<a\leq 1$ if $C$ is TPI and $-1\leq a<0$ if $C$ is TNI.
Moreover, the conditional distribution function $$G_{+}(x)=\Prob(X\leq x\mid X\geq 0)=2G(x)-1,\quad x> 0,$$ satisfies
\begin{align}\label{eq:G:plus:identity}
\E_{G_{+}}[Z]\geq |a|^{1/2},\quad Z\sim G_{+}\quad \text{and}\quad
G_{+}(x)=1-G_{+}\left(\frac{|a|}{x}-\right),\quad x> 0.
\end{align}
In particular, it holds that $\Prob(Z>|a|^{1/2})=\Prob(Z<|a|^{1/2})$ for $Z \sim G_{+}$.
\end{enumerate}
\end{proposition}

By Proposition~\ref{prop:necessary:conditions:g:attain:optimal:best:asymptotic:variance},
not all $G \in \mathcal G_4$ and $C \in \mathcal C_2$ can attain the optimal best asymptotic variance $\sigma_G^2(C)=0$.
The following examples show non-Bernoulli concordance-inducing distributions attaining this lower bound.
Let $M(n,\{J_i\},\pi,w)$ denote a shuffle-of-$M$ with $n$ being the number of connected components in its support, $\{J_i\}=\{J_1,\dots,J_n\}$ being a finite partition of $[0,1]$ into $n$ closed subintervals, $\pi$ being a permutation of $\{1,\dots,n\}$ and $w:\{1,\dots,n\}\rightarrow \{-1,1\}^n$ being a function indicating whether the strip $J_i\times J_{\pi(i)}$ is flipped ($w(i)=1$) or not ($w(i)=-1$); see \cite[Section~3.2.3]{nelsen2006introduction}.

\begin{example}[Non-Bernoulli concordance-inducing distributions in $\underline G_\ast(\mathcal G_4,\mathcal D)$]\label{example:non:bernoulli:optimal:best:g:functions}
\hspace{2mm}
\begin{enumerate}
\item The case when $\Prob(X=0)>0$: Let $X\sim G$ be an equally weighted mixture of $0$ and $\Unif(-\sqrt{6},\sqrt{6})$. Then $\E[X]=0$, $\Var(X)=1$ and $\E[X^4]<\infty$, and thus $G \in \mathcal G_4$.
This is Case~\ref{item:necessary:conditions:point:mass:case} of Proposition~\ref{prop:necessary:conditions:g:attain:optimal:best:asymptotic:variance} since $\Prob(X=0)=1/2$.
Consider $C_1=M(4,\cup_{i=1}^4[(i-1)/4,i/4],\{2,1,4,3\},\bone_4)$, $C_2=M(4,\cup_{i=1}^4[(i-1)/4,i/4],\{3,4,1,2\},\bone_4)$ and $C_3=M(4,\cup_{i=1}^4[(i-1)/4,i/4],\{2,4,1,3\},\bone_4)$, where $\bone_d=(1,\dots,1)\in\IR^d$.
Then $C_1$ is TPI, $C_2$ is TNI and $C_3$ is neither TPI nor TNI.
Moreover, they satisfy $\sigma_G^2(C_k)=0$, $k=1,2,3$, since $G^\i(U)G^\i(V)\aseq 0$ with $(U,V)\sim C_k$ for $k=1,2,3$.
\item The case when $\Prob(X=0)=0$:
Let $X\sim G$ be a discrete uniform distribution on the four points $\{-a/b,-b,b,a/b\}$ where $a=1/2^{1/2}$ and 
$b=(1-2^{-1/2})^{1/2}$ with $b \approx 0.541$ and $a/b \approx 1.307$.
Then it is straightforward to check that $G \in \mathcal G_4$.
Define $(X,Y)=(G^\i(U),G^\i(V))$ with $(U,V)\sim C_4= M(4,\cup_{i=1}^4[(i-1)/4,i/4],\{2,1,4,3\},-\bone_4)$.
Then $(X,Y)=(-a/b,-b)$, $(-b,-a/b)$, $(b,a/b)$ and $(a/b,b)$ are equiprobable, and thus  $\sigma_G^2(C_4)=0$ since $XY\aseq a$.
This case belongs to Case~\ref{item:necessary:conditions:no:point:mass:case} of Proposition~\ref{prop:necessary:conditions:g:attain:optimal:best:asymptotic:variance} since $C_4$ is TPI, $0< a \leq 1$ and $\E_{G_{+}}[Z]\approx 0.924 > 0.841 \approx \sqrt{a}$.
\end{enumerate}
\end{example}

\section{Comparison of $\kappa_G$ and Kendall's tau}\label{sec:comaprison:kendall:tau}

In Section~\ref{subsec:optimality:Blomqvist}, we showed that Blomqvist's beta is an optimal $G$-transformed rank correlation under some conditions on $\mathcal D\subseteq \mathcal C_2$.
In this section, we show that Kendall's tau, although it is not a transformed rank correlation, shares certain optimal structure of Blomqvist's beta.

Kendall's tau $\tau:\mathcal C_2 \rightarrow \IR$ is defined by
\begin{align}\label{eq:kendall:tau}
\tau(C)=4\int_{[0,1]^2}C(u,v)\,\rd C(u,v)-1,
\end{align}
and is a measure of concordance; see \cite{scarsini1984}.
Moreover, it is not a $G$-transformed rank correlation since $\tau$ is not linear with respect to a mixture of copulas.
Since $\tau(C)=\rho(\bone_{\{U> \tilde U\}},\bone_{\{V> \tilde V\}})$ where $(U,V)\sim C$ and $(\tilde U,\tilde V)\sim C$ is an independent copy of $(U,V)$, $\tau$ can also be written as
\begin{align}\label{eq:kendall:representation:correlation}
\tau(C) = \rho(g(U,\tilde U),g(V,\tilde V)),\quad 
g(\ell,m)=\begin{cases}
1 & \text{ if } \ell> m ,\\
-1&  \text{ if }\ell \leq  m,\\
\end{cases}
\end{align}
by invariance of $\rho$ under location-scale transforms.
Assuming that the data-generating i.i.d.\,process from $C$ is available, we consider the following estimator of $\tau(C)$ to estimate a $G$-transformed rank correlation $\kappa_{G}$ based on the representation~\eqref{eq:kendall:representation:correlation}.

\begin{definition}[Canonical estimator of $\tau$]
\label{def:canonical:estimator:tau}
For $C \in \mathcal C_2$, the canonical estimator of $\tau$ is given by
\begin{align*}
\hat \tau=\hat \tau(C;n)=\frac{1}{n}\sum_{i=1}^n   g(U_i,\tilde U_i)g(V_i,\tilde V_i),\quad
(U_1,V_1),\dots,(U_n,V_n),
(\tilde U_1,\tilde V_1),\dots, (\tilde U_n,\tilde V_n) \iidsim C,\quad n\in \IN.
\end{align*}
\end{definition}

By the classical central limit theorem, the following asymptotic normality follows
\begin{align*}
\sqrt{n}\left\{\hat \tau - \tau(C)\right\} \darrow \N(0,\sigma_\tau^2(C)),
\quad  n \rightarrow \infty,
\end{align*}
where the asymptotic variance of $\hat \tau$ is given by
$$\sigma_\tau^2(C)=\Var(g(U,\tilde U)g(V,\tilde V)),\quad C \in \mathcal C_{2}.$$

For $(U,V),\  (\tilde U,\tilde V)\iidsim C$, write $X=g(U,\tilde U)$ and $Y=g(V,\tilde V)$.
As discussed in Section~\ref{subsec:optimal:location:shift}, location shift of $X$ and $Y$ does not change $\tau(C)=\rho(X,Y)$ but may affect $\sigma_\tau^2(C)$.
Thus we define the optimal location shift as follows.

\begin{definition}[Optimal shift of $g$]
For $C\in \mathcal C_2$, the optimal (location) shift of $g$ under $C$ is given by
\begin{align*}
\mu_{\ast} = \mu_{\ast}(\tau,C)=\begin{cases}
-\frac{\Cov(XY,\, X+Y)}{\Var(X+Y)}, & \text{ if } \Var(X+Y)>0,\\
0, & \text{ if } \Var(X+Y)=0.
\end{cases}
\end{align*}
where $X=g(U,\tilde U)$ and $Y=g(V,\tilde V)$ for $(U,V),\  (\tilde U,\tilde V)\iidsim C$.
\end{definition}

Basic properties of $\mu_{\ast}$ and $\sigma_\tau^2$ are collected in the next proposition.

\begin{proposition}[Basic properties of $\mu_{\ast}$ $\sigma_\tau^2(C)$]
\label{prop:basic:properties:tau:asymptotic:variance}
Let $C \in \mathcal C_2$.
Then the following properties hold for $\tau$:
\begin{enumerate}
\item\label{item:optimal:shift:tau:asymptotic:variance}
$\mu_\ast(\tau,C)=0$;
\item\label{item:explicit:form:tau:asymptotic:variance}
$\sigma_{\tau}^2(C)=4p_\tau(C)(1-p_\tau(C))=1-\tau^2(C)$
where
$p_\tau(C)=2\int_{[0,1]^2}C(u,v)\,\rd C(u,v)$;
\item\label{item:bounds:kendall:asymptotic:variance}
$0\leq \sigma_\tau^2(C) \leq 1$;
\item\label{item:lower:bound:kendall:asymptotic:variance}
$\sigma_{\tau}^2(C)=0$ if and only if $\tau(C)=1$ or $-1$, that is, $C=M$ or $W$, respectively;
\item\label{item:upper:bound:kendall:asymptotic:variance}
$\sigma_{\tau}^2(C)=1$ if and only if $\tau(C)=0$.
In particular, $\tau(C)=0$ when $C=C_{\nu_1}$ or $C=C_{\nu_2}$.
\end{enumerate}
\end{proposition}

\begin{remark}[Asymptotic variance of $\tau$ for elliptical copulas]\label{remark:explicit:forms:asymptotic:variances:Kendall:tau}
When $C$ is an elliptical copula with correlation parameter $\rho\in[-1,1]$, we have that $\tau(C)=2\arcsin (\rho)/\pi$; see \cite{hult2002multivariate}.
Therefore, by Proposition~\ref{prop:asymptotic:variance:kendall:tau} Part~\ref{item:explicit:form:tau:asymptotic:variance}, we have that
$\sigma_{\tau}^2(C)=1-\left(2\arcsin (\rho)/\pi\right)^2$, which also equals $\sigma_{\beta}^2(C)$ as derived in Remark~\ref{remark:asymptotic:variances:Blomqvist:beta:elliptical}.
\end{remark}

Similar to the case of $G$-transformed rank correlations, we consider the following best and worst asymptotic variances.

\begin{definition}[Best and worst asymptotic variances for $\tau$]
For $\mathcal D \subseteq C_2$, the best and worst asymptotic variances are defined by
\begin{align*}
\underline \sigma_{\tau}^2(\mathcal D)=\inf_{C \in \mathcal D}\sigma_{\tau}^2(C),
\quad
\overline \sigma_{\tau}^2(\mathcal D)=\sup_{C \in \mathcal D}\sigma_{\tau}^2(C),
\end{align*}
respectively.
If the infimum and supremum above are attainable, then their attaining copulas on $\mathcal D$ are defined, respectively, by
\begin{align*}
\underline C_{\tau}(\mathcal D)=\argmin_{C \in \mathcal D}\sigma_{\tau}^2(C)
\quad
\overline C_{\tau}(\mathcal D)=\argmax_{C \in \mathcal D}\sigma_{\tau}^2(C).
\end{align*}
\end{definition}

Properties of the best and worst asymptotic variances of $\tau$ are summarized as follows.

\begin{proposition}[Best and worst asymptotic variances of Kendall's tau]\label{prop:asymptotic:variance:kendall:tau}\hspace{2mm}
Let $\mathcal D \subseteq C_2$ and $\beta \in \mathcal H \subseteq \mathcal G_4$.
\begin{enumerate}
\item\label{item:attain:lower:bound:kendall:asymptotic:variance}
If $M \in \mathcal D$ or $W \in \mathcal D$, then $\underline \sigma_{\tau}^2(\mathcal D)=\underline\sigma_{\ast}^2(\mathcal H,\mathcal D)=0$ and $\underline C_{\tau}(\mathcal D)=\{M,W\}\cap \mathcal D$.
\item\label{item:attain:upper:bound:kendall:asymptotic:variance}
If $\Pi \in \mathcal D$, then $\overline \sigma_{\tau}^2(\mathcal D)=\overline\sigma_{\ast}^2(\mathcal H,\mathcal D)=1$ and $\Pi \in \overline C_{\tau}(\mathcal D)$.
\item \label{itemattain:upper:bound:kendall:asymptotic:variance:Frechet:copulas}
If $\mathcal D= \mathcal C^{\text{F}}$, then $\overline \sigma_{\tau}^2(\mathcal D)=1$ and
$\overline C_{\tau}(\mathcal D)= \left\{ p (M+W)/2 + (1-p)\Pi: p \in [0,1]\right\}$.
\end{enumerate}
\end{proposition}

Proposition~\ref{prop:asymptotic:variance:kendall:tau}
states that Kendall's tau attains the optimal best and worst asymptotic variances of transformed rank correlations, which are also attained by Blomqvist's beta.
Taking into account the drawback of Blomqvist's beta that it depends only on the local value $C(1/2,1/2)$ of a copula $C$, Kendall's tau can be a good alternative of Blomqvist's beta in terms of best and worst asymptotic variances.

\begin{remark}[Comparability of $\kappa_G$ and $\tau$]
\label{remark:comparability:kappa:g:tau}
Since Representation~\eqref{eq:kendall:representation:correlation} of Kendall's tau in terms of Pearson's correlation coefficient depends on two independent copies $(U,V)\sim C$ and $(\tilde U,\tilde V)\sim C$, the canonical estimator $\hat \tau$ requires twice more samples from $C$ than $\hat \kappa$ does.
Therefore, if the estimators $\hat \tau$ and $\hat \kappa_G$ are compared based on their actual variances (instead of their asymptotic variances), then $\Var(\hat \tau)=\sigma_{\tau}^2(C)/n$ should be multiplied by $2$ to be compared with $\Var(\hat \kappa_G)=\sigma_G^2(C)/n$.
With this modification, Kendall's tau still attains the optimal best asymptotic variance $\underline\sigma_{\ast}^2(\mathcal G_4,\mathcal D)=0$ since $2 \underline \sigma_{\tau}^2(\mathcal D)=0$.
On the other hand, Kendall's tau fails to attain the optimal worst asymptotic variance $\overline\sigma_{\ast}^2(\mathcal G_4,\mathcal D)=1$ since
$2 \overline \sigma_{\tau}^2(\mathcal D)=2>1$.
\end{remark}

\section{Simulation study}\label{sec:simulation:study}

In this section, we conduct a simulation study to compare the asymptotic variance $\sigma_G^2(C)$ for various copulas $C \in \mathcal C_2$ and concordance-inducing distributions $G\in \mathcal G_4$.
Not only Spearman's rho $\rho_{\text{S}}$, Blomqvist's beta $\beta$ and van der Waerden's coefficient $\zeta$, we also consider $G$-transformed rank correlations with $G$ given by a Student $t$ distribution $t(\nu)$ with $\nu=10$ degrees of freedom and a beta distribution with shape parameters $(0.5,0.5)$. 
Note that both are radially symmetric and have finite fourth moments, and thus belong to $\mathcal G_4$ after standardization (mean zero and variance one).
The Beta$(0.5,0.5)$ concordance-inducing distribution has a different shape from the others since it puts an increasing probability mass as locations farther away from the center $1/2$.
Kendall's tau is also considered for comparison.
Besides standardized concordance-inducing distributions, we also consider optimally shifted ones as introduced in Section~\ref{subsec:optimal:location:shift}.
As underlying copulas, we consider Gaussian $C_{\rho}^{\text{Ga}}$, Student $t$ $C_{\rho,\nu}^{t}$ and Clayton copulas $C_{\theta}^{\text{Cl}}$ where $\rho \in [-1,1]$ is a correlation parameter, $\nu>0$ is a degree of freedom and $\theta\geq -1$ is a shape parameter.
The experiment consists of the following three steps.

\begin{enumerate}\setlength{\itemsep}{-0mm}
\item\label{item:step:set:parameters:copulas} Set $\rho = -0.99 +1.98 k / 49$ for $k\in \{ 0,1,\dots,49\}$, $\nu=5$ and $\theta = 2\rho/(1-\rho)$ (which yields $\tau(C_{\theta}^{\text{Cl}})=\rho$) in $C=C_{\rho}^{\text{Ga}}$, $C_{\rho,\nu}^{t}$ and $C_{\theta}^{\text{Cl}}$.
\item\label{item:step:simulate:sample:C} For each copula $C$ in Step~\ref{item:step:set:parameters:copulas}, simulate $(U_1,V_1),\dots,(U_n,V_n)\iidsim C$ with $n=10^5$.
\item\label{item:step:compute:estimates:asymptotic:variances} Based on the samples generated in Step~\ref{item:step:simulate:sample:C}, estimate $\sigma_G^2(C)$ and $\sigma_{\tau}^2(C)$ by the sample variances of $G^\i(U_i)G^\i(V_i)$, $i=1,\dots,n$, and of $g(U_i,U_{i+n/2})g(V_i,V_{i+n/2})$, $i=1,\dots,n/2$, where $G$ is a standardized, and optimally shifted uniform, Beta$(0.5,0.5)$, normal, $t(10)$ and symmetric Bernoulli distribution function, and $g$ is as defined in~\eqref{eq:kendall:representation:correlation}.
\end{enumerate}

\begin{figure}[t]
  \centering
  \vspace{0mm}
  \includegraphics[width=14 cm]{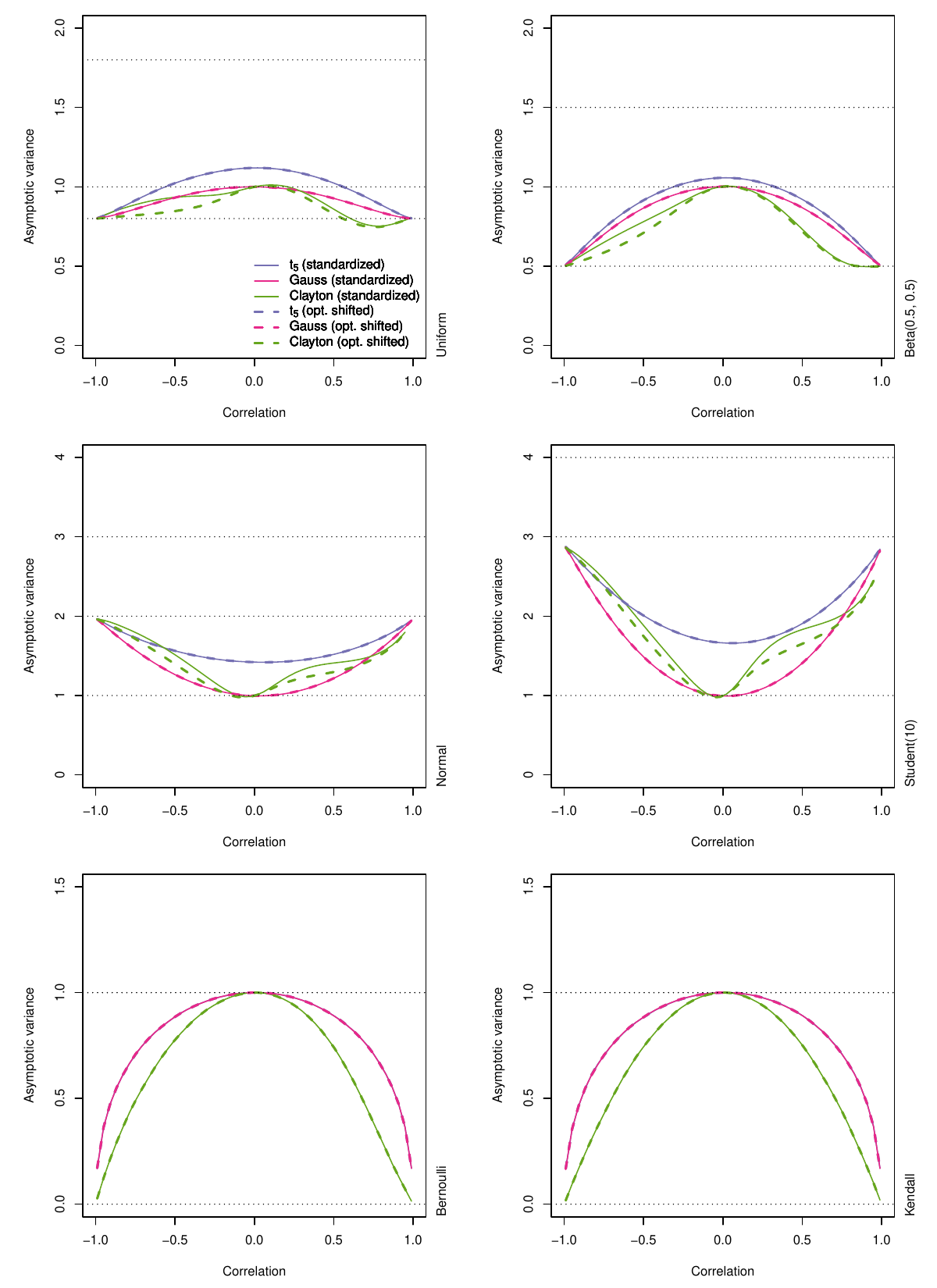}
  \vspace{0mm}
  \caption{Estimates of asymptotic variances $\sigma_G^2(C)$ and $\sigma_{\tau}^2(C)$ against correlation parameters $\rho \in [-0.99,0.99]$ of $C=C_{\rho}^{\text{Ga}}$ (red), $C_{\rho,\nu}^{t}$ (blue) with $\nu=5$ and $C_{\theta}^{\text{Cl}}$ (green) with $\theta = 2\rho/(1-\rho)$ for $G$-transformed rank correlation coefficients $\kappa_G$ (all except bottom-right) and Kendall's tau $\tau$ (bottom-right).
  The concordance-inducing distribution $G$ is set to be standardized (solid lines) and optimally shifted (dashed lines) uniform, Beta$(0.5,0.5)$, normal, $t(10)$ and symmetric Bernoulli distribution.
  The black dotted lines represent $y=1$, $V(G)$ and $1+V(G)$ with $V(\beta)=0$, $V(\rho_{\text{S}})=0.8$, $V(\zeta)=2$, $V(t(10))=3$, $V(\text{Beta(0.5,0.5)})=0.5$ and $V(\tau)=0$.}
  \label{fig:asymptotic:variances}
\end{figure}

The estimates of $\sigma_G^2(C)$ and $\sigma_{\tau}^2(C)$ computed in Step~\ref{item:step:compute:estimates:asymptotic:variances} are plotted in Fig.~\ref{fig:asymptotic:variances}.
Observations from these plots are summarized as follows.
For $G\in\mathcal G$, we denote by $V(G)$ (or $V(\kappa_G)$) the variance of $X^2$ where $X$ follows the standarzied distribution of $G$ so that $\E[X]=0$ and $\Var(X)=1$.
In addition, we write $V(\tau)=\Var(g(U,\tilde U))$ for $U,\, \tilde U \iidsim \Unif(0,1)$.

\emph{Symmetry}:
For all copulas $C$, the curves of $\sigma_G^2(C)$ and $\sigma_{\tau}^2(C)$ against the correlation parameter $\rho$ were almost symmetric around $\rho=0$.
For $C=C_{\rho}^{\text{Ga}}$ and $C_{\rho,\nu}^{t}$, the symmetry is a consequence from Proposition~\ref{prop:basic:properties:asymptotic:variance} Part~\ref{item:reflection:invariance:asymptotic:variance} since
$C_{-\rho}^{\text{Ga}}=\nu_1(C_{\rho}^{\text{Ga}})$ and $C_{-\rho,\nu}^{t}=\nu_1(C_{\rho,\nu}^{t})$.
This argument does not apply to Clayton copulas, and thus the curves $\rho \mapsto \sigma_G^2(C_{2\rho/(1-\rho)}^{\text{Cl}})$ and $\rho \mapsto \sigma_{\tau}^2(C_{2\rho/(1-\rho)}^{\text{Cl}})$ are not precisely symmetric.

\emph{Convexity and concavity}:
At least in this experiment, the curves of $\sigma_G^2(C)$ and $\sigma_{\tau}^2(C)$ are convex when $V(G)>1$ (which holds if $G$ is normal or $t(10)$), and concave when $V(G)<1$ (which holds if $G$ is Bernoulli, uniform and Beta$(0.5,0.5)$, and if Kendall's tau is considered).
This observation is consistent with~\eqref{eq:optimal:asymptotic:variance:normal:variance:mixture:variance:X:squared}, and with the asymptotic variances obtained in Remarks~\ref{remark:asymptotic:variances:Blomqvist:beta:elliptical} and~\ref{remark:explicit:forms:asymptotic:variances:Kendall:tau}.

\emph{Best and worst asymptotic variances}:
For all cases of $C=C_{\rho}^{\text{Ga}}$, $C_{\rho,\nu}^{t}$ and $C_{\theta}^{\text{Cl}}$, the best and worst asymptotic variances were approximately $1 \wedge V(G)$ and $1\vee V(G)$, respectively.
When $V(G)>1$ (normal or $t(10)$), the lower bound $\sigma_G^2(C)=1$ was attained at $\rho=0$ and the upper bound $\sigma_G^2(C)=V(G)$ was attained at $\rho=\pm 1$.
When $V(G)<1$ (Bernoulli, uniform, Beta$(0.5,0.5)$ and Kendall), the lower bound $\sigma_G^2(C)=V(G)$ was attained at $\rho=\pm 1$ and the upper bound $\sigma_G^2(C)=1$ was attained at $\rho=0$.
Note that, for $C=C_{\rho,\nu}^{t}$, the curve was slightly above these lower and upper bounds at $\rho=0$ since $C_{0,\nu}^{t}\neq \Pi$.
Since we only consider specific classes of copulas, the global upper bound $1+V(G)$ as presented in Proposition~\ref{prop:basic:properties:asymptotic:variance} Part~\ref{item:bounds:asymptotic:variance} was not attained except in the cases of Blomqvist's beta and Kendall's tau where $V(\beta)=V(\tau)=0$.

\emph{Choice of $G$; normal or Student, and uniform or Beta}:
As seen for the best and worst asymptotic variances, the variance $V(G)$ may play important roles in determining the best and worst asymptotic variances.
As theoretically indicated, concordance-inducing distributions with smaller $V(G)$ are more preferable at least in terms of asymptotic variance.
Therefore, the normal concordance-inducing distribution is more preferable to $t(10)$ since $V(\zeta)=2 < 3 = V(t(10))$.
In fact, for all copulas considered, $\zeta$ had a smaller asymptotic variance than $\kappa_{t(10)}$.
Interestingly, $\kappa_{\text{Beta}(0.5,0.5)}$ typically had smaller asymptotic variance than $\rho_{\text{S}}$ as expected from $V(\text{Beta})=0.5 < 0.8 = V(\rho_{\text{S}})$.

\emph{Blomqvist's beta and Kendall's tau}:
As indicated in Remark~\ref{remark:explicit:forms:asymptotic:variances:Kendall:tau}, the curves of $\sigma_\beta^2(C)$ and $\sigma_\tau^2(C)$ seem to coincide when $C=C_{\rho}^{\text{Ga}}$ and $C_{\rho,\nu}^{t}$.
On the other hand, $\sigma_{\beta}^2(C_{\theta}^{\text{Cl}})$ and $\sigma_{\tau}^2(C_{\theta}^{\text{Cl}})$ are in general different since
$\sigma_{\beta}^2(C_{\theta}^{\text{Cl}})=1-\beta^2(C_{\theta}^{\text{Cl}})$ and $\sigma_{\tau}^2(C_{\theta}^{\text{Cl}})=1-\tau^2(C_{\theta}^{\text{Cl}})$ by Proposition~\ref{prop:asymptotic:variances:beta} Part~\ref{item:explicit:form:asymptotic:variances:beta} and Proposition~\ref{prop:basic:properties:tau:asymptotic:variance} Part~\ref{item:explicit:form:tau:asymptotic:variance}, but $\beta(C_{\theta}^{\text{Cl}})=4(2^{\theta+1}-1)^{-1/\theta}-1$ and $\tau(C_{\theta}^{\text{Cl}})=\theta/(\theta+2)$.

\emph{Strength and the model of dependence}:
For any concordance-inducing distribution, the difference of $\sigma_G^2(C)$ among different copulas $C=C_{\rho}^{\text{Ga}}$, $C_{\rho,\nu}^{t}$ and $C_{\theta}^{\text{Cl}}$ was typically smaller than the difference of $\sigma_G^2(C)$ among different levels of dependence, which is controlled by $\rho$ in this experiment.
Therefore, one might say that the choice of $C$ is less influential on the asymptotic variance $\sigma_G^2(C)$ compared with the choice of concordance-inducing distribution and the strength of dependence.

\emph{Effect of optimal shifts}:
As theoretically indicated in Proposition~\ref{prop:sufficient:condition:optimal:shift:zero}, Proposition~\ref{prop:asymptotic:variances:beta} Part~\ref{item:optimal:shift:asymptotic:variances:beta} and Proposition~\ref{prop:basic:properties:tau:asymptotic:variance} Part~\ref{item:optimal:shift:tau:asymptotic:variance},
the solid and dotted curves of asymptotic variances overlapped when $C=C_{\rho}^{\text{Ga}}$ or $C_{\rho,\nu}^{t}$, and when $\beta$ and $\tau$ are considered.
For other cases, the optimal shift reduced the asymptotic variance.
However, even when the copula is $C_{\theta}^{\text{Cl}}$, only a small reduction by the optimal shift was observed in this experiment.

\section{Concluding remarks and discussion}\label{sec:concluding:remark}

To answer the question which measure of concordance is best to use, we proposed a comparison of $G$-transformed rank correlations $\kappa_G$ in terms of their best and worst asymptotic variances $\sigma_G^2(C)$ on a set of copulas $\mathcal D$.
When $\mathcal D$ is a set of fundamental copulas or Fr\'echet copulas, we showed that the best and worst asymptotic variances are increasing functions of $\Var_G(X^2)$, $X\sim G$, which leads to the comparison criterion that concordance-inducing distributions $G$ having smaller $\Var_G(X^2)$ are more preferable.
Since $\Var_G(X^2)$ attains its minimum $\Var_G(X^2)=0$ if and only if $G$ is a standardized symmetric Bernoulli distribution,
we proved that Blomqvist's beta $\beta$ uniquely attains the optimal best and worst asymptotic variances among all transformed rank correlations under certain conditions on $\mathcal D$.
Considering the drawback of $\beta$ that it depends only on the local value $C(1/2,1/2)$ of a copula $C$, we also compared $G$-transformed rank correlations with Kendall's tau $\tau$.
Based on the representation of $\tau$ in terms of Pearson's linear correlation coefficient, we found that $\tau$ also attains the optimal best and worst asymptotic variances that $\beta$ does, although $\tau$ is not a transformed rank correlation.
Since the estimator of $\tau$ requires twice more samples than that of $\kappa_G$ does, some optimality of $\tau$ is violated if this required sample size is taken into account.
In a simulation study, we investigated the curve of the asymptotic variance of $\kappa_G$ and $\tau$ against the strength of dependence of the underlying copula.
We observed that the curve of $\sigma_G^2(C)$ was typically symmetric and convex or concave with the best and worst asymptotic variances give by $1\wedge \Var_G(X^2)$ and $1\vee \Var_G(X^2)$ depending on $\Var_G(X^2)>1$ or $\Var_G(X^2)<1$.
These observations support the criterion that concordance-inducing distributions $G$ with smaller $\Var_G(X^2)$ are more preferable.
Consequently, heavy-tailed concordance-inducing distributions, such as Student $t$ distributions with small degrees of freedom, are not recommended in comparison to the normal distribution, which leads to van der Waerden's coefficient $\zeta$.
In addition, we found that the beta distribution-based transformed rank correlations can be good alternatives to Spearman's rho $\rho_{\text{S}}$.

Below we list limitations, discussion and future research on this work.
\begin{itemize}
\item Although Blomqvist's beta $\beta$ can be
optimal in terms of stability of its statistical estimation, this measure has some limitations. 
For example, $|\beta(C)|=1$ does not imply $C =M,\, W$; see~\cite[Proposition~1]{koike2022matrix}.
\item   Given the
limitations of Fr\'echet copulas in practice, it may be interesting to explore
optimal concordance-inducing distributions under more practical choices of sets
of the underlying copulas $\mathcal D$, such as a set of parametric copulas or a
ball of copulas around a given reference copula.
\item In our work we assumed for simplicity that i.i.d.\ samples from the underlying
copula are available. However, this may be unrealistic in practice, and it is
therefore of interest whether and how results in this paper change if pseudo-samples from the underlying copula are used in the comparison of measures of concordance in terms of their asymptotic variances.
\item
Besides Kendall's tau, there are still important measures of concordance, such as Gini's gamma, which are not
transformed rank correlations. It is thus of interest to
study a broader framework that allows one to include such measures of
concordance in comparison.
In particular, for a comparison, it may be better to consider other
estimators of Kendall's tau than the canonical one considered in this paper due to the required
sample size.
\end{itemize}

\section{Proofs}\label{sec:proofs}


\begin{proof}[\bf Proof of Proposition~\ref{prop:basic:properties:asymptotic:variance}]
\hspace{0mm}\\
\noindent \ref{item:covariance:formula:asymptotic:variance})
For $(U,V)\sim C$, write $X=G^\i(U)$ and $Y=G^\i(V)$.
Since $G \in \mathcal G_4$, we have that 
\begin{align*}
\Cov(X^2,Y^2)=\E[(XY)^2]-\E[X^2]\E[Y^2]=\E[(XY)^2]-1,\quad \Cov(X,Y)=\E[XY]-\E[X]\E[Y]=\E[XY].
\end{align*}
Therefore,
\begin{align*}
\sigma_G^2(C)=\Var(XY)=\E[(XY)^2]-\E[XY]^2=\Cov(X^2,Y^2)+1 - \Cov(X,Y)^2.
\end{align*}

\noindent \ref{item:fundamental:copulas:asymptotic:variance})
If $(U,V)\sim \Pi$, then $(X^2,Y^2)$ and $(X,Y)$ are both independent random vectors, and thus $\Cov(X^2,Y^2)=\Cov(X,Y)=0$.
Therefore $\sigma_G^2(\Pi)=0 + 1- 1^2 = 1$ by~\eqref{eq:covariance:formula:asymptotic:variance}.
If $C=M$, then we have that $\Cov(X,Y)=\kappa_G(M)=1$ and that $(X^2,Y^2)\deq(G^\i(U)^2,G^\i(U)^2)$ for $U\sim \Unif(0,1)$.
When $C=W$, we have that $\Cov(X,Y)=\kappa_G(W)=-1$ and that
\begin{align*}
(X^2,Y^2)\deq(G^\i(U)^2,G^\i(1-U)^2)=(G^\i(U)^2,(-G^\i(U))^2)=(G^\i(U)^2,G^\i(U)^2),\quad U\sim \Unif(0,1),
\end{align*}
by radial symmetry of $G \in \mathcal G_4$.
Therefore, in either case, we have $\rho(X^2,Y^2)=1$ and thus
\begin{align*}
\Cov(X^2,Y^2)= \rho(X^2,Y^2)\operatorname{SD}_G(X^2)\operatorname{SD}_G(Y^2)=\Var_G(X^2).
\end{align*}
Therefore, by~\eqref{eq:covariance:formula:asymptotic:variance}, we have that
$$\sigma_G^2(M) =\Var_G(X^2)+1-1^2=\Var_G(X^2),\quad \sigma_G^2(W) =\Var_G(X^2)+1-(-1)^2=\Var_G(X^2).$$

\noindent \ref{item:bounds:asymptotic:variance})
By Cauchy--Schwarz inequality, it holds that $\Cov(X^2,Y^2)\leq \operatorname{SD}(X^2)\operatorname{SD}(Y^2)=\Var(X^2)$.
Since $\Cov(X,Y)^2\geq 0$, the desired inequalities hold by~\eqref{eq:covariance:formula:asymptotic:variance}.
When $C=(M+W)/2$, we have that
\begin{align*}
(X,Y)\deq B  (G^\i(U),G^\i(U))+ (1-B)(G^\i(U),G^\i(1-U)),
\end{align*}
where $B \sim \Bern(1/2)$ and $U \sim \Unif(0,1)$ independent of $B$.
Therefore,
\begin{align*}
\Cov(X,Y)=\frac{1}{2}\Cov(G^\i(U),G^\i(U)) + \frac{1}{2}\Cov(G^\i(U),G^\i(1-U))=\frac{1}{2}(1)+\frac{1}{2}(-1)=0,
\end{align*}
and $\Cov(X^2,Y^2)=\Cov(G^\i(U)^2,G^\i(U)^2)=\Var_G(X^2)$.
Using~\eqref{eq:covariance:formula:asymptotic:variance}, we have that $$\sigma_G\left((M+W)/2\right)=\Var_G(X^2)+1-0^2=\Var_G(X^2)+1.$$

\noindent \ref{item:boundedness:asymptotic:variance})
Since the fourth moment of $G$ is finite, we have that $\Var_G(X^2)<\infty$.
Therefore $\sigma_G^2(C)<\infty$ by~\ref{item:bounds:asymptotic:variance}).

\noindent \ref{item:reflection:invariance:asymptotic:variance})
The desired result follows from~\eqref{eq:covariance:formula:asymptotic:variance} since radial symmetry of $G \in \mathcal G$ implies that
\begin{align}\label{eq:almost:sure:equalities:radial:symmetry}
(G^\i(U),G^\i(V))&\aseq
(-G^\i(U_{\nu_1}),G^\i(V_{\nu_1}))\aseq
(G^\i(U_{\nu_2}),-G^\i(V_{\nu_2}))
\aseq(-G^\i(U_{\nu_1\circ \nu_2}),-G^\i(V_{\nu_1 \circ \nu_2})).
\end{align}
\end{proof}

\begin{proof}[\bf Proof of Proposition~\ref{prop:concavity:sigma:G:C:asymptotic:variance}]

For $C, C' \in \mathcal C_2$ and $p \in [0,1]$, define the random vector $(\tilde X,\tilde Y)=(G^\i(\tilde U),G^\i(\tilde V))$ where $(\tilde U,\tilde V)= B (U,V) + (1-B)(U',V')$, $(U,V)\sim C$, $(U',V')\sim C'$ and $B\sim \Bern(p)$ is independent of $(U,V)$ and $(U',V')$.
Then $(\tilde U,\tilde V)\sim\tilde C_p$ where $\tilde C_p=pC + (1-p)C'$.
Moreover, we have that
\begin{align*}
(\tilde X,\tilde Y)=(G^\i(\tilde U),G^\i(\tilde V))=(G^\i(BU+(1-B)U'),G^\i(BV+(1-B)V'))=B(X,Y)+(1-B)(X',Y'),
\end{align*}
where  $(X,Y)=(G^\i(U),G^\i(V))$ and $(X',Y')=(G^\i(U'),G^\i(V'))$.
From this representation, it holds that
\begin{align*}
\Cov(\tilde X^2,\tilde Y^2)=p\Cov(X^2,Y^2)+(1-p)\Cov({X'}^2,{Y'}^2),\quad
\Cov(\tilde X,\tilde Y)=p\Cov(X,Y)+(1-p)\Cov(X',Y').
\end{align*}
Therefore, we have that
\begin{align*}
\sigma_G^2(\tilde C_p)& =\Var(\tilde X \tilde Y)= \Cov(\tilde X^2,\tilde Y^2)+ 1 -\Cov(\tilde X,\tilde Y)^2\\
&= p\Cov(X^2,Y^2)+(1-p)\Cov({X'}^2,{Y'}^2) + 1 - (p\Cov(X,Y)+(1-p)\Cov(X',Y'))^2\\
&\geq  p\Cov(X^2,Y^2)+(1-p)\Cov({X'}^2,{Y'}^2) + 1 - p\Cov(X,Y)^2-(1-p)\Cov(X',Y')^2\\
&= p \Var(XY)+(1-p)\Var(X'Y')=p\sigma_G^2(C)+(1-p)\sigma_G^2(C'),
\end{align*}
where the inequality in the third line holds since
\begin{align*}
p\Cov(X,Y)^2+(1-p)\Cov(X',Y')^2 &- (p\Cov(X,Y)+(1-p)\Cov(X',Y'))^2\\
&= p(1-p)(\Cov(X,Y)- \Cov(X',Y'))^2 \geq0.
\end{align*}
Therefore, the map $C \mapsto \sigma_G^2(C)$ is concave.
When $C, C' \in \mathcal C_G(k)$ for some $k \in [-1,1]$, we have $\Cov(X,Y)=\Cov(X',Y')$, and thus equality holds in the inequality above.
Consequently, the map $C \mapsto \sigma_G^2(C)$ is linear.
\end{proof}

\begin{proof}[\bf Proof of Proposition~\ref{prop:sufficient:condition:optimal:shift:zero}]

By definition of $\mu_{\ast}$, it suffices to consider the case when $\Var(X_0+Y_0)>0$ for $X_0 = G_0^\i(U)$ and $Y_0 = G_0^\i(V)$ with $(U,V)\sim C$.

Since $\E[X_0+Y_0]=\E[X_0]+\E[Y_0]=0$, we have that 
\begin{align*}
\Cov(X_0Y_0,X_0+Y_0)=\E[X_0Y_0(X_0+Y_0)]-\E[X_0Y_0]\E[X_0+Y_0]=\E[X_0Y_0(X_0+Y_0)].
\end{align*}
Therefore, it suffices to show that $\E[X_0Y_0(X_0+Y_0)]=0$ when $C$ is radially symmetric.

When $C$ is radially symmetric, we have that $(U,V)\deq (U_{\nu_1\circ \nu_2},$ $V_{\nu_1\circ \nu_2})$ and $(U_{\nu_1},V_{\nu_1})\deq (U_{\nu_2},V_{\nu_2})$ for $(U,V)\sim C$.
Together with the identity
\begin{align*}
1&=\bone_{\{U>1/2,\,V>1/2\}} + \bone_{\{U\leq 1/2,\,V>1/2\}} + \bone_{\{U>1/2,\,V\leq 1/2\}} + \bone_{\{U\leq 1/2,\,V\leq 1/2\}}\\
&= \bone_{\{U>1/2,\,V>1/2\}} + \bone_{\{U_{\nu_1}>1/2,\,V_{\nu_1}>1/2\}} + \bone_{\{U_{\nu_2}>1/2,\,V_{\nu_2}>1/2\}} + \bone_{\{U_{\nu_1\circ \nu_2}>1/2,\,V_{\nu_1\circ \nu_2}>1/2\}},
\end{align*}
we have, by~\eqref{eq:almost:sure:equalities:radial:symmetry}, that
\begin{align*}
\E[X_0Y_0  (X_0+Y_0)]
  &= \sum_{\varphi  \in \{\iota, \nu_1,\nu_2, \nu_1\circ \nu_2\}}\E[\bone_{\{U_{\varphi}>1/2,V_{\varphi}>1/2\}}\,G_0^\i(U)G_0^\i(V)\,( G_0^\i(U)+ G_0^\i(V))]\\
&=  \E[\bone_{\{U>1/2,V>1/2\}}\,G_0^\i(U)G_0^\i(V)\,( G_0^\i(U)+ G_0^\i(V))]\\
& \hspace{5mm}- \E[\bone_{\{U_{\nu_1\circ \nu_2}>1/2,V_{\nu_1\circ \nu_2}>1/2\}}\,G_0^\i(U_{\nu_1\circ \nu_2})\,G_0^\i(V_{\nu_1\circ \nu_2})\,( G_0^\i(U_{\nu_1\circ \nu_2})+ G_0^\i(V_{\nu_1\circ \nu_2}))]\\
&\hspace{5mm}+ \E[\bone_{\{U_{\nu_1}>1/2,V_{\nu_1}>1/2\}}\,G_0^\i(U_{\nu_1})G_0^\i(V_{\nu_1})\,( G_0^\i(U_{\nu_1})-G_0^\i(V_{\nu_1}))]\\
&\hspace{5mm}- \E[\bone_{\{U_{\nu_2}>1/2,V_{\nu_2}>1/2\}}\,G_0^\i(U_{\nu_2})G_0^\i(V_{\nu_2})\,( G_0^\i(U_{\nu_2})- G_0^\i(V_{\nu_2}))]=0,
\end{align*}
which gives the desired result $\mu_{\ast}(G,C)=0$.
\end{proof}

\begin{proof}[\bf Proof of Proposition~\ref{prop:asymptotic:variance:normal:variance:mixture}]

Since normal variance mixtures are radially symmetric, the optimal asymptotic variance is obtained when $\mu^{\ast}=0$.
Since $\sigma_{G_{\mu^\ast}}^2(C)$ is invariant under scale transforms of $G_{\mu_{\ast}}$, we standardize $(X,Y)$ to apply the covariance formula~\eqref{eq:covariance:formula:asymptotic:variance}.
Namely, we set $\Sigma=P/\E[W]$ where $P$ is a correlation matrix with off-diagonal entry $\rho=\rho(X,Y)$.
In this setup, we have that
\begin{align*}
\Var(X)=\Var(Y)=\E[WZ_1^2]-\E\left[\sqrt{W}Z_1\right]^2=\E[W]\Var(Z_1)-0=\E[W]\,\frac{1}{\E[W]}=1.
\end{align*}
Since $\Cov(Z_1^2,Z_2^2)=2\rho^2/\E[W]^2$, the law of total covariance implies that
\begin{align*}
\Cov(X^2,Y^2)&=\E[\Cov(X^2,Y^2)\ |\ W)] + \Cov(\E[X^2\ |\ W],\ \E[Y^2\ |\ W]) \\
&= \E[W^2\Cov(Z_1^2,Z_2^2)] + \Cov(\E[Z_1^2]W,\E[Z_2^2]W)=\E[W^2]\Cov(Z_1^2,Z_2^2) + \Var(W)\Var(Z_1)\Var(Z_2)\\
&= 2\rho^2\frac{\E[W^2]}{\E[W]^2}+\frac{\Var(W)}{\E[W]^2}.
\end{align*}
Therefore, we have, by~\eqref{eq:covariance:formula:asymptotic:variance}, that
\begin{align*}
\sigma_{G_{\mu^\ast}}^2(C)&=\Var(XY)=\Cov(X^2,Y^2)+1-\Cov(X,Y)^2
=2\rho^2\frac{\E[W^2]}{\E[W]^2}+\frac{\Var(W)}{\E[W]^2}+ 1 - \rho^2,
\end{align*}
and thus the desired result follows.
\end{proof}

\begin{proof}[\bf Proof of Proposition~\ref{prop:worst:best:asymptotic:variances:Frechet:copulas}]

Fix $G \in \mathcal G_4$ and $C_{\bp}^{\operatorname{F}} \in \mathcal C^{\operatorname{F}}$ with $\bp=(p_M,p_{\Pi},p_W) \in \Delta_3$.
For $X=G^\i(U)$ and $Y=G^\i(V)$ with $(U,V)\sim C_{\bp}^{\operatorname{F}}$, we have that $\Cov(X^2,Y^2)=(p_M + p_W)\Var_G(X^2)$ and that $\Cov(X,Y)=p_M-p_W$.
Therefore, by~\eqref{eq:covariance:formula:asymptotic:variance}, it holds that
\begin{align*}
\sigma_G^2(C_{\bp}^{\operatorname{F}})=(p_M + p_W)v + 1 - (p_M-p_W)^2 =:f(p_M,p_W),
\end{align*}
where $v=\Var_G(X^2)$ for notational convenience.
Since the Hessian of $f$
\begin{align*}
H(p_M,p_W)=\begin{pmatrix}
\frac{\partial}{\partial p_M^2}f(p_M,p_W)& \frac{\partial}{\partial p_M p_W}f(p_M,p_W)
\\
\frac{\partial}{\partial p_W p_M}f(p_M,p_W) & \frac{\partial}{\partial p_W^2}f(p_M,p_W)\\
\end{pmatrix}
= \begin{pmatrix}
-2& 2\\
2& -2\\
\end{pmatrix},
\end{align*}
is nonpositive definite, $f$ is a concave function.

For $(p_M,p_W) \in \IR^2$ such that $p_M,p_W\geq 0$ and $p_M+p_W\leq 1$, consider the reparametrization $(p,0) + r(-1,1)=(p-r,r)$ where $0\leq r \leq p \leq 1$.
Then
\begin{align*}
f(p-r,r)=p v +1-(p-r)^2-r^2+2(p-r)r=-4\left(r-\frac{p}{2}\right)^2 + p v + 1,
\end{align*}
and thus $f$ represents a parabolic cylinder.

For a fixed $p \in [0,1]$, the function $r \mapsto f(p-r,r)$ has a maximum $\overline f(p)=p v+1$ when $r=p/2$, and a minimum $\underline f(p)=-p^2+p v+1$ when $r=0$ or $r=p$.
Since $v\geq 0$, the maximum of $f$ is given by $v+1$ with the maximum attained by $p=1$ when $v>0$, and by any $p\in [0,1]$ when $v=0$.
Therefore, we have that $\overline \sigma_G^2(\mathcal C^{\operatorname{F}})=v+1 = \sigma_G^2(C)$ with $C=(M+W)/2$ when $v>0$, and with $C=p \left(M+W\right)/2 + (1-p)\Pi$ for any $p \in [0,1]$ when $v=0$.
For the minimum of $f$, notice that the function $\underline f(p)=-p^2+p v+1$, $0\leq p \leq 1$, is a concave parabola, and thus the minimum of $\underline f(p)$ is attained at $p=0$ or $p=1$.
With $\underline f(0)=1$ and $\underline f(1)=v$, the minimum of $f$ and its attainers are given by $\underline \sigma_G^2(\mathcal C^{\operatorname{F}})=1\wedge v = \sigma_G^2(C)$ with $C=M$ or $W$ when $0\leq v < 1$, with $C=M$, $\Pi$ or $W$ when $v=1$ and with $C=\Pi$ when $v>1$.
\end{proof}

\begin{proof}[\bf Proof of Corollary~\ref{cor:preference;orders:popular:mocs}]
For $G\in\mathcal G$, denote by $V(G)$ (or $V(\kappa_G)$) the variance of $X^2$ where $X$ follows the standardized distribution of $G$ so that $\E[X]=0$ and $\Var(X)=1$.
Then the concordance-inducing distributions of $\rho_{\text{S}}$, $\beta$ and $\zeta$ are the uniform distribution on $\left(-\sqrt{3},\sqrt{3}\right)$, the symmetric Bernoulli distribution on $\{-1,1\}$ and the standard normal distribution $N(0,1)$.
Since
\begin{align*}
V(\rho_{\text{S}})=V(\Unif(0,1))=0.8,\quad
V(\beta)=V(\Bern(1/2))=0\quad\text{and}\quad
V(\zeta)=V(\N(0,1))=2,
\end{align*}
the result follows from Corollary~\ref{cor:optimal:concordance:inducing:functions:fundamental:frechet:copulas} Part~\ref{item:preference:order:funndamental:frechet:copuas}.
\end{proof}

\begin{proof}[\bf Proof of Proposition~\ref{prop:asymptotic:variances:beta}]
Let $G$ be the standardized symmetric Bernoulli distribution.

\noindent \ref{item:optimal:shift:asymptotic:variances:beta})
For $(X,Y)=(G^\i(U),G^\i(V))$ with $(U,V)\sim C$, it suffices to consider the case when $\Var(X+Y)>0$ by definition of $\mu_\ast$.
Since $G^\i(u)=2\bone_{\{u>1/2\}}-1$, $u \in (0,1)$, we have that
\begin{align*}
X+Y&=
\begin{cases}
2, & \text{ if }\  \{U >  1/2,\ V >  1/2\},\\
-2, & \text{ if }\ \{U \leq  1/2,\ V \leq  1/2\},\\
0, & \text{ if }\  \{U > 1/2,\ V \leq 1/2\}\cup \{U \leq  1/2,\ V >  1/2\},
\end{cases}\\
XY&=\begin{cases}
1, & \text{ if }\  \{U \leq 1/2,\ V \leq 1/2\}\cup \{U >  1/2,\ V >  1/2\},\\
-1, & \text{ if }\  \{U > 1/2,\ V \leq 1/2\}\cup \{U \leq  1/2,\ V >  1/2\}.
\end{cases}
\end{align*}
Since $C(1/2,1/2)=\bar C(1/2,1/2)$, we have that $\E[X+Y]=2\bar C(1/2,1/2)-2 C(1/2,1/2)=0$, and thus
\begin{align*}
\Cov(X+Y,XY)&=\E[(X+Y)XY] - \E[X+Y]\E[XY] = 2\bar C(1/2,1/2)-2 C(1/2,1/2)-0 = 0,
\end{align*}
which implies that $\mu_\ast(\beta,C)=0$.

\noindent \ref{item:explicit:form:asymptotic:variances:beta})
Using the notation
\begin{align*}
p(C)=\Prob\left(\{U \leq 1/2,\ V \leq  1/2\}\cup \{ U >  1/2,\ V >  1/2\}\right)=C(1/2,1/2)+\bar C(1/2,1/2)=2C(1/2,1/2),
\end{align*}
we have that $\beta(C)=4C(1/2,1/2)-1=2p(C)-1$ and that
\begin{align*}
\sigma_\beta^2(C)=\Var(XY)=4 p(C) (1-p(C))=1-\beta^2(C).
\end{align*}

\noindent \ref{item:upper:lower:bounds:asymptotic:variances:beta})
This immediately follows from \ref{item:explicit:form:asymptotic:variances:beta}) since $0 \leq p(C)\leq 1$.

\noindent \ref{item:lower:bound:asymptotic:variances:beta})
By \ref{item:explicit:form:asymptotic:variances:beta}), $\sigma_\beta^2(C)=0$ if and only if $p(C)=0$ or $1$, that is, $C$ is a TPI or TNI copula.

\noindent \ref{item:upper:bound:asymptotic:variances:beta})
By \ref{item:explicit:form:asymptotic:variances:beta}), $\sigma_\beta^2(C)=1$ if and only if $p(C)=1/2$, that is, $C$ is a balanced copula.

\end{proof}

\begin{proof}[\bf Proof of Proposition~\ref{prop:optimality:blomqvist:beta}]

\noindent \ref{item:optimality:blomqvist:beta:lower})
By Proposition~\ref{prop:asymptotic:variances:beta} Part~\ref{item:lower:bound:asymptotic:variances:beta}, we have that
\begin{align}\label{eq:key:inequalities:optimality:blomqvist:beta:lower}
\underline \sigma_\beta^2(\mathcal D)=\sigma_\beta^2(C_\ast)=0 \leq \underline \sigma_G^2(\mathcal D)\quad \text{for all } G \in \mathcal G_4,
\end{align}
and thus $\underline \sigma_\ast^2(\mathcal H,\mathcal D)=0$ and $\beta \in \underline G_\ast(\mathcal H,\mathcal D)$.

\noindent \ref{item:optimality:blomqvist:beta:upper})
By Proposition~\ref{prop:basic:properties:asymptotic:variance} Part~\ref{item:fundamental:copulas:asymptotic:variance} and Proposition~\ref{prop:asymptotic:variances:beta} Parts~\ref{item:upper:lower:bounds:asymptotic:variances:beta} and~\ref{item:upper:bound:asymptotic:variances:beta}, we have that
\begin{align}\label{eq:key:inequalities:optimality:blomqvist:beta:upper}
\overline \sigma_\beta^2(\mathcal D)=\sigma_\beta^2(\Pi)=1 = \sigma_G^2(\Pi) \leq \overline \sigma_G^2(\mathcal D)\quad \text{for all } G \in \mathcal G_4,
\end{align}
and thus $\overline \sigma_\ast^2(\mathcal H,\mathcal D)=1$ and $\beta \in \overline G_\ast(\mathcal H,\mathcal D)$.

\noindent \ref{item:optimality:blomqvist:beta})
$\beta \in G_\ast(\mathcal H,\mathcal D)$ immediately follows from \ref{item:optimality:blomqvist:beta:lower}) and \ref{item:optimality:blomqvist:beta:upper}).

Finally, all the relations in~\eqref{eq:key:inequalities:optimality:blomqvist:beta:lower} and~\eqref{eq:key:inequalities:optimality:blomqvist:beta:upper} remain valid even if $\sigma_G^2(C)$, $G \in \mathcal G_4$, is replaced by the optimally shifted asymptotic variance $\sigma_{G_{\mu_{\ast}}}^2(C)$.
Therefore, all the optimality results in
\ref{item:optimality:blomqvist:beta:lower}),
\ref{item:optimality:blomqvist:beta:upper}) and
\ref{item:optimality:blomqvist:beta})
hold for correspondingly modified versions of Definitions~\ref{Def:optimal:best:worst:asymptotic:variances} and~\ref{def:preference:optimal:concordance:inducing:distribution}.

\end{proof}

\begin{proof}[\bf Proof of Proposition~\ref{prop:uniqueness:beta:worst:G}]

By Proposition~\ref{prop:optimality:blomqvist:beta}, we have that $\beta \in G_\ast(\mathcal H,\mathcal D)$.
Proposition~\ref{prop:basic:properties:asymptotic:variance} Part~\ref{item:bounds:asymptotic:variance} also yields
$\overline \sigma_G^2(\mathcal D)=1+\Var_G(X^2)$ for all $G \in \mathcal H$.
This upper bound remains valid if $G \in \mathcal H$ is optimally shifted since $(M+W)/2$ is radially symmetric.
Therefore, regardless of whether the optimal shift is taken into account, the optimal worst asymptotic variance $\overline \sigma_{\ast}^2(\mathcal H,\mathcal D)= 1$ is attained if and only if $G \in \mathcal H$ satisfies $\Var_G(X^2)=0$, that is, $X\sim G$ is the standardized symmetric Bernoulli distribution.
Consequently, we have that $\overline G_{\ast}(\mathcal H,\mathcal D)=\{\beta\}$ and thus $G_{\ast}(\mathcal H,\mathcal D)=\{\beta\}$ as desired.
\end{proof}

\begin{proof}[\bf Proof of Proposition~\ref{prop:necessary:conditions:g:attain:optimal:best:asymptotic:variance}]
For $G \in \mathcal H$ and $C \in \mathcal D$ in~\eqref{eq:equivalent:conditions:attaining:lower:bound:asymptotic:variance:bernoulli}, write $(X,Y)=(G^\i(U),G^\i(V))$.

\noindent \ref{item:necessary:conditions:point:mass:case})
In this case, we have that $\Prob(XY=0)>0$ and thus $a \in \IR$ in~\eqref{eq:equivalent:conditions:attaining:lower:bound:asymptotic:variance:bernoulli} necessarily has to be $a=0$.
If $XY\aseq 0$ holds, then $Y=0$ on $\{X\neq 0\}$.
Together with $X\deq Y$, we have that
$\Prob(X\neq 0)\leq \Prob(Y=0)=\Prob(X=0)$,
which leads to the condition $\Prob(X=0)\geq 1/2$.

\noindent \ref{item:necessary:conditions:no:point:mass:case})
In this case, we have that $\Prob(XY=0)=0$ since
\begin{align*}
XY  \begin{cases}
 >0, & \text{ if } \{U\leq 1/2,\ V \leq 1/2\}\cup \{U> 1/2,\ V > 1/2\},\\
 <0,  & \text{ if } \{U\leq 1/2,\ V > 1/2\}\cup \{U> 1/2,\ V \leq  1/2\}.\\
\end{cases}
\end{align*}
Therefore, $a\in \IR$ in~\eqref{eq:equivalent:conditions:attaining:lower:bound:asymptotic:variance:bernoulli} necessarily has to be unequal to $0$.
Since $$\Prob(XY>0)=p(\{U\leq 1/2,\ V \leq 1/2\}\cup \{U> 1/2,\ V > 1/2\})=p(C)$$ and $$\Prob(XY<0)=p(\{U\leq 1/2,\ V > 1/2\}\cup \{U> 1/2,\ V \leq  1/2\})=1-p(C),$$ the random variable $XY$ can never be almost surely a constant if $0<p(C)<1$.
Therefore, it holds that $p(C)=0$ or $1$, and thus $C$ is either TPI or TNI.

Assume that $C$ is TPI.
Then $a>0$ since $\Prob(XY>0)=1$.
By the TPI assumption of $C$, we have that
 \begin{align*}
 X_{+}&=X\ | \ \{U>1/2,V>1/2\}=X\ | \ \{U>1/2\} \sim G_{+},\\
 Y_{+}&=Y\ | \ \{U>1/2,V>1/2\}=Y\ | \ \{V>1/2\} \sim G_{+},\\
 X_{-}&=X\ | \ \{U\leq 1/2,V\leq 1/2\}=X\ | \ \{U\leq 1/2\} \sim G_{-},\\
 Y_{-}&=X\ | \ \{U\leq 1/2,V\leq 1/2\}=Y\ | \ \{V\leq 1/2\} \sim G_{-},
 \end{align*}
 where
\begin{align*}
G_{+}(x)=\begin{cases}
2G(x)-1, & \text{ if } x > 0,\\
0, & \text{ if } x \leq  0,\\
\end{cases}
\quad
G_{-}(x)=\begin{cases}
1, & \text{ if } x > 0,\\
2G(x), & \text{ if } x \leq  0.\\
\end{cases}
\end{align*}
In addition to the equalities $X_{+}\deq Y_{+}$ and $X_{-}\deq Y_{-}$, we have that $X_{+}\deq - X_{-}$ and $Y_{+}\deq - Y_{-}$ since the radial symmetry of $G$ and the condition $\Prob(X=0)=0$ lead to
\begin{align*}
\Prob(-X_{-}\leq x)&= \Prob(X_{-}\geq -x)=1-G_{-}((-x)-)\\
&=\begin{cases}
1-1 = 0, & \text{ if } x < 0,\\
1-2G((-x)-)=1-2(1-G(x))=2G(x)-1,& \text{ if } x \geq 0,
\end{cases}\\&= G_{+}(x).
\end{align*}
 Moreover, since $XY\aseq a$, it holds that
 \begin{align*}
 X_{+}Y_{+}=XY\ | \ \{U>1/2,V>1/2\}\aseq a \quad\text{and}\quad
 X_{-}Y_{-}=XY\ | \ \{U\leq1/2,V\leq 1/2\}\aseq a.
 \end{align*}
Since $X_{+}Y_{+}\aseq a$ and $Y_{+}>0$ a.s., we have that $X_{+} \aseq a/Y_{+}$.
Therefore, Jensen's inequality implies that
\begin{align*}
\E[X_{+}]=\E\left[\frac{a}{Y_{+}}\right]=a \E\left[\frac{1}{Y_{+}}\right]\geq \frac{a}{\E[Y_{+}]}=\frac{a}{\E[X_{+}]},
\end{align*}
which yields $\E[X_{+}]\geq \sqrt{a}$.

Since $X_{+}\deq - X_{-}$ and $\Var(X)=\E[X^2]=1$, we have that
\begin{align*}
1&=\E[X^2]=\Prob\left(U>\frac{1}{2}\right)\E\left[X^2\ \biggl|\ U>\frac{1}{2}\right] +\Prob\left(U\leq \frac{1}{2}\right)\E\left[X^2\ \biggl|\ U\leq \frac{1}{2}\right]
=\frac{1}{2}\E[X_{+}^2] + \frac{1}{2}\E[X_{-}^2] =\E[X_{+}^2],
\end{align*}
and thus $\E[X_{+}^2]=1$.
Using $X_{+}^2 \aseq (a/Y_{+})^2>0$ a.s.\ and Jensen's inequality, we have that
\begin{align*}
1=\E[X_{+}^2] =\E\left[\left(\frac{a}{Y_{+}}\right)^2\right]\geq \frac{a^2}{\E[Y_{+}^2]}=\frac{a^2}{\E[X_{+}^2]},
\end{align*}
which yields $-1\leq a\leq 1$.
Together with $a>0$, we have the inequalities $0<a\leq 1$.
Moreover, $X_{+} \aseq a/Y_{+}$ implies that
\begin{align*}
G_{+}(x)=\Prob(X_{+}\leq x)=\Prob\left(\frac{a}{Y_{+}}\leq x\right)=1-\Prob\left(Y_{+}<\frac{a}{x}\right)=1-G_{+}\left(\frac{a}{x}-\right),
\quad x>0,
\end{align*}
which leads to the identity~\eqref{eq:G:plus:identity}.
The symmetry $\Prob(Z>a^{1/2})=\Prob(Z<a^{1/2})$ for $Z \sim G_{+}$ is obtained as a special case by taking $x=\sqrt{a}>0$ in~\eqref{eq:G:plus:identity}.

Next assume that $C$ is TNI.
Then $a<0$ since $\Prob(XY<0)=1$.
By the TNI assumption, we have that
 \begin{align*}
 X_{+}&=X\ | \ \{U>1/2,V \leq 1/2\}=X\ | \ \{U>1/2\} \sim G_{+},\\
 Y_{+}&=Y\ | \ \{U\leq 1/2,V>1/2\}=Y\ | \ \{V>1/2\} \sim G_{+},\\
 X_{-}&=X\ | \ \{U\leq 1/2,V> 1/2\}=X\ | \ \{U\leq 1/2\} \sim G_{-},\\
 Y_{-}&=X\ | \ \{U > 1/2,V\leq 1/2\}=Y\ | \ \{V\leq 1/2\} \sim G_{-}.
 \end{align*}
 As in the TPI case, it holds that $X_{+}\deq Y_{+}$, $X_{-}\deq Y_{-}$, $X_{+}\deq - X_{+}$ and $Y_{+}\deq - Y_{+}$.
 Moreover, $XY\aseq a$ implies that
 \begin{align*}
 X_{+}Y_{-}=XY\ | \ \{U>1/2,V\leq 1/2\}\aseq a,
\quad
 X_{-}Y_{+}=XY\ | \ \{U\leq1/2,V >  1/2\}\aseq a.
 \end{align*}
 From these equalities, all the necessary conditions derived in the TPI case hold with $-Y_{-},\ -X_{-} \sim G_{+}$ and $-a>0$ since $X_{+}(-Y_{-})\aseq (-X_{-})Y_{+} \aseq  -a$.
\end{proof}

\begin{proof}[\bf Proof of Proposition~\ref{prop:basic:properties:tau:asymptotic:variance}]
For $(U,V),\  (\tilde U,\tilde V)\iidsim C$, write $X=g(U,\tilde U)$ and $Y=g(V,\tilde V)$.

\noindent \ref{item:optimal:shift:tau:asymptotic:variance})
The statement holds when $\Var(X+Y)=0$.
Assume that $\Var(X+Y)>0$.
Then we have that
\begin{align*}
X+Y&=
\begin{cases}
2, & \text{ if }\  \{U >  \tilde U,\ V >  \tilde V\},\\
-2, & \text{ if }\ \{U \leq  \tilde U,\ V \leq  \tilde V\},\\
0, & \text{ if }\  \{U > \tilde U,\ V \leq \tilde V\}\cup \{U \leq  \tilde U,\ V >  \tilde V\},
\end{cases}\\
XY&=\begin{cases}
1, & \text{ if }\  \{U \leq \tilde U,\ V \leq \tilde V\}\cup \{U >  \tilde U,\ V >  \tilde V\},\\
-1, & \text{ if }\  \{U > \tilde U,\ V \leq \tilde V\}\cup \{U \leq  \tilde U,\ V >  \tilde V\}.
\end{cases}
\end{align*}
Then $\E[X+Y]=0$ and $\Cov(X+Y,XY)=0$ by calculation since
\begin{align*}
\Prob(U >  \tilde U,\ V >  \tilde V)
=\int_{[0,1]^2}C(u,v)\,\rd C(u,v).
\end{align*}
Therefore, we have that $\mu_\ast(\tau,C)=0$ as desired.

\noindent \ref{item:explicit:form:tau:asymptotic:variance})
By definition of Kendall's tau~\eqref{eq:kendall:tau}, we have that
\begin{align*}
\Prob(\{U\leq \tilde U,\ V\leq \tilde V\}\cup \{U> \tilde U,\ V> \tilde V\})
= 2\int_{[0,1]^2}C(u,v)\,\rd C(u,v)=p_\tau(C)=\frac{\tau(C)+1}{2},
\end{align*}
and thus $\sigma_{\tau}^2(C)=\Var(XY)=4p_\tau(C)(1-p_\tau(C))=1-\tau^2(C)$.

\noindent \ref{item:bounds:kendall:asymptotic:variance})
This is an immediate consequence from \ref{item:explicit:form:tau:asymptotic:variance}) and $0\leq p_{\tau}(C)\leq 1$.

\noindent \ref{item:lower:bound:kendall:asymptotic:variance})
$\sigma_{\tau}^2(C) = 0$ is attained if and only if $p_\tau(C)=1$ or $0$, that is, $\tau(C)=1$ or $-1$, respectively.
By \cite[Theorem~3]{embrechtsmcneilstraumann2002}, $\tau(C)=1$ or $-1$ if and only if $C=M$ or $W$, respectively.

\noindent \ref{item:upper:bound:kendall:asymptotic:variance})
$\sigma_{\tau}^2(C) = 1$ is attained if and only if $p_\tau(C)=1/2$, that is, $\tau(C)=0$.
When $C=C_{\nu_1}$ or $C=C_{\nu_2}$, the change of sign axiom of measures of concordance in Definition~\ref{def:axioms:MOC} implies that $\tau(C)=\tau(C_{\nu_1})=-\tau(C)$ or $\tau(C)=\tau(C_{\nu_2})=-\tau(C)$, either of which yields $\tau(C)=0$.
\end{proof}

\begin{proof}[\bf Proof of Proposition~\ref{prop:asymptotic:variance:kendall:tau}]
\ref{item:attain:lower:bound:kendall:asymptotic:variance}) and \ref{item:attain:upper:bound:kendall:asymptotic:variance}) immediately follow from Proposition~\ref{prop:basic:properties:tau:asymptotic:variance} Parts~\ref{item:bounds:kendall:asymptotic:variance},~\ref{item:lower:bound:kendall:asymptotic:variance} and~\ref{item:upper:bound:kendall:asymptotic:variance}, and Proposition~\ref{prop:optimality:blomqvist:beta}.

\noindent \ref{itemattain:upper:bound:kendall:asymptotic:variance:Frechet:copulas})
By Proposition~\ref{prop:basic:properties:tau:asymptotic:variance} Part~\ref{item:upper:bound:kendall:asymptotic:variance}, it holds that $\sigma_\tau^2(C)=1$ for $C\in\mathcal C_2$ if and only if $\tau(C)=0$.
For a Fr\'echet copula, we have that $\tau(\mathcal C_{(p_M,p_{\Pi},p_W)}^{\text{F}})=(p_M-p_W)(p_M+p_W+2)/3$; see \cite[Example~5.3]{nelsen2006introduction}.
Therefore, $\tau(\mathcal C^{\text{F}})=0$ holds if and only if $p_M=p_W$, and thus the desired result follows.
\end{proof}

\section*{Acknowledgments}
We are grateful to Alexander Schied and Ruodu Wang at University of Waterloo for their valuable comments.
Takaaki Koike was supported by JSPS KAKENHI Grant Number JP21K13275.
Marius Hofert acknowledges financial support from the Natural Sciences and Engineering Research Council of Canada (RGPIN-2020-04897 and RGPAS-2020-00093).

\bibliographystyle{myjmva}

\end{document}